\documentclass{amsart}
\usepackage{amsfonts,amscd, amssymb}
\usepackage{hyperref}
\usepackage{xcolor}
\usepackage{tikz-cd}
\usetikzlibrary{matrix}
\newtheorem{theorem}{Theorem}[section]
\newtheorem{lemma}[theorem]{Lemma}
\newtheorem{corollary}[theorem]{Corollary}
\newtheorem{proposition}[theorem]{Proposition}
\theoremstyle{remark}

\theoremstyle{definition}
\newtheorem{definition}[theorem]{Definition}
\newtheorem{example}[theorem]{Example}
\numberwithin{equation}{section} \makeatother

\DeclareMathOperator{\Cdb}{{\mathbb C}}
\DeclareMathOperator{\Ddb}{{\mathbb D}}

\DeclareMathOperator{\Kdb}{{\mathbb K}}

\DeclareMathOperator{\Mdb}{{\mathbb M}}
\DeclareMathOperator{\Ndb}{{\mathbb N}}

\DeclareMathOperator{\Rdb}{{\mathbb R}}

\DeclareMathOperator{\Tdb}{{\mathbb T}}
\DeclareMathOperator{\Zdb}{{\mathbb Z}}

\DeclareMathOperator{\Fl}{{\mathcal F}}

\DeclareMathOperator{\Ll}{{\mathcal L}}
\DeclareMathOperator{\Ml}{{\mathcal M}}

\DeclareMathOperator{\Sl}{{\mathcal S}}

\DeclareMathOperator{\Ul}{{\mathcal U}}

\DeclareMathOperator{\re}{{Re}}

\DeclareMathOperator{\ball}{Ball}

\newcommand{\osa}[1]{\operatorname{oa^{\ast}}(#1)}
\newcommand{\oa}[1]{\operatorname{oa}(#1)}

\newcommand{\dg}{\dagger}

\begin{document}

\title[Involutive operator algebras]{Involutive operator algebras}

\author{David P. Blecher}
\address{Department of Mathematics, University of Houston, Houston, TX
77204-3008}
\email[David Blecher]{dblecher@math.uh.edu}
\author{Zhenhua Wang}
\address{Department of Mathematics, University of Houston, Houston, TX
77204-3008}
\email[Zhenhua Wang]{zhenwang@math.uh.edu}
\date{Revision of February 11, 2019}
 \thanks{We acknowledge
support by Simons Foundation grant  527078}
 \maketitle
\begin{abstract}
  Examples of operator algebras with involution include
the operator $*$-algebras occurring in noncommutative differential geometry
studied recently by Mesland, Kaad, Lesch, and others, several classical function 
algebras, triangular matrix algebras, 
(complexifications) of real operator algebras, and an operator 
algebraic version of the {\em complex symmetric} operators studied
by Garcia, Putinar, Wogen, Zhu, and others.   We investigate the general theory of involutive operator 
algebras, and give  many  applications, such as a characterization of  the symmetric operator algebras introduced in the 
early days of operator space theory.   
\end{abstract} 

\section{Introduction}
An {\em operator algebra} for us is a closed subalgebra of $B(H),$ for a complex Hilbert space $H.$ 
Here we study operator algebras with involution.  Examples include
the operator $*$-algebras occurring in noncommutative differential geometry
studied recently by Mesland, Kaad, Lesch, and others (see
e.g.\ \cite{Mes14,KL,BKM} and references therein),
(complexifications) of real operator algebras, and an operator 
algebraic version of the {\em complex symmetric} operators studied
by Garcia, Putinar, Wogen, Zhu, and many others (see \cite{GP}  for a survey, or e.g.\  \cite{GW}).  

By an {\em operator $\ast$-algebra} we mean an operator algebra with an involution $\dagger$ making it a $\ast$-algebra 
with $\Vert [a_{ji}^{\dagger}]\Vert=\Vert [a_{ij}]\Vert$ 
for $[a_{ij}] \in M_n(A)$ and $n \in \Ndb$.   Here we are using the 
{\em matrix norms} of operator space theory (see e.g.\ \cite{Pisbk}).
This notion was first introduced by Mesland 
in the setting of noncommutative differential geometry \cite{Mes14},
who was soon joined by  Kaad and Lesch \cite{KL}.  
In several recent  papers by these authors and coauthors they 
exploit operator $\ast$-algebras and involutive modules  in geometric situations.   
Subsequently we noticed very many other examples of 
operator $\ast$-algebras, and other 
involutive operator
algebras, occurring naturally in general 
operator algebra theory which 
 seem to have not been studied hitherto.  
It is thus natural to investigate the general theory of involutive operator 
algebras, and this is the focus
of the present paper.  
We are able to include a rather large number of results
 since many proofs are similar to their
operator algebra counterparts in 
the  literature (see e.g.\ \cite{BLM}).     Thus we often need only discuss the
new points that arise.  However
to  follow some of the arguments the reader  will need to have the operator algebra variant
from the original papers in hand.     It is worth saying that some of the arguments
we are following are complicated, and so it is not clear in advance whether they have  `involutive variants'.    In fact
 some of the main theorems about 
operator algebras do not have operator $*$-algebra variants,  
so some work is needed to
disentangle the items that do work. 
We make no attempt to be comprehensive for the sake of avoiding
tedium.   We will simply 
illustrate the main techniques and features, 
indicating what can be done.   Many of the results are focused around 
`real positivity' in the sense of several recent papers of the first 
author and collaborators referenced in our bibliography. 
 Some related theory and several complementary results can be found in the second authors  PhD thesis \cite{Wang}.

\subsection{Structure of our paper} In the rest of this section 
we give some background, perspective, and  notations.  In Section 2 we give several general
results.  For example we prove some facts about involutions
on nonselfadjoint operator algebras and their relationship
to the $C^*$-algebras they generate.
As an application of some ideas 
in the theory of complex symmetric operators 
we characterize the symmetric operator algebras introduced
in \cite{BComm}.  This is a problem outstanding from the early years of operator space theory. Section 3 is 
devoted to examples of involutive operator algebras,
for instance examples coming from operator space 
theory, subdiagonal algebras, model theory for contractions on a Hilbert space,
and complex symmetric operators.

In the remaining sections 
we restrict our focus, for specificity, to
operator $\ast$-algebras.    Many of our results involve {\em real positivity} in the sense of several recent papers of the first 
author and collaborators referenced in our bibliography, where the cones $\Rdb_+  {\mathfrak F}_A$ and ${\mathfrak r}_A$ (defined below) act as a substitute 
for the usual positive cone in a $C^*$-algebra.    In an operator $*$-algebra $A$, ${\mathfrak F}_A$ and ${\mathfrak r}_A$ are closed under the involution,
and taking $n$th roots commutes with the involution.  Thus the theory of real positivity studied in many of the first authors recent papers will have good involutive variants.   
 In Section 4 we discuss contractive approximate identities, Cohen factorization for operator $\ast$-algebras, multiplier operator $\ast$-algebras,
dual operator $\ast$-algebras (by which
we mean an operator $*$-algebra which is a
dual operator space with
weak* continuous
involution), and involutive $M$-ideals.  
Section 5 has a common theme of hereditary subalgebras and
ideals,  noncommutative topology (e.g.\ open projections, 
support projections, and compact and peak projections),  and peak interpolation, in the involutive setting.  Thus we are finding the involutive variants of 
the operator algebra theory of these topics
from e.g.\ the papers \cite{BRI, BRII, BRord, BNII, BHN, Bnpi}. 
 
  \subsection{Involutions, and notation} \label{inv} 

By an involution we mean at least a bijection $\tau : A \to A$ which is of
 period 2: $\tau^2(a) = a$ for $a \in A$.
A $C^*$-algebra $B$ may have two kinds of extra involution: a period 2 conjugate linear $*$-antiautomorphism or a period 2  linear $*$-antiautomorphism.
The former is just the usual involution $*$ composed with a period 2  $*$-automorphism of $B$.  The latter is essentially the same as a `real structure', that is
if $\theta$ is the antiautomorphism
then  $B$ is just the complexification of a real $C^*$-algebra $D = \{ x \in B : x = \bar{x} \}$, where $\bar{x} = \theta(x)^*$.
We may characterize $x \mapsto \bar{x}$ on $B$ very simply as the map $a + i b \mapsto a - ib$ for $a, b \in D$.    

By way of contrast, there are four distinct natural 
kinds of `completely isometric involution' on a general operator algebra $A$.  Namely, period 2 bijections which are 
\begin{itemize}
\item [(1)]  conjugate linear  antiautomorphisms $\dagger : A \to A$ satisfying $\| [a_{ji}^\dagger ] \|
= \| [ a_{ij}] \| ,$
\item [(2)]  
  linear antiautomorphisms $\theta : A \to A$ satisfying $\| [a_{ji}^\theta ] \|
= \| [ a_{ij}] \| ,$
\item [(3)]  conjugate linear automorphisms $\textendash : A \to A$ satisfying $\| [\overline{a_{ij}} ] \|
= \| [ a_{ij}] \| ,$
\item [(4)]   linear automorphisms $\pi : A \to A$ satisfying $\| [a_{ij}^\pi ] \|
= \| [ a_{ij}] \| .$
\end{itemize}
 Here $[a_{ij}]$ is a generic element in $M_n(A)$, the $n \times n$ matrices with entries in $A$, for all $n \in \Ndb$.  
Class (1) is just the operator $*$-algebras mentioned earlier.
In this paper we will call the algebras in class (2) {\em  operator algebras with linear involution} $\theta$, and write $\theta(a)$
 as $a^\theta$.
We will not discuss (4) in this paper, these are well studied and are only mentioned here because most of the results in the present paper apply
to all four classes.  We will just say that this class is in bijective correspondence with the unital {\em completely symmetric projections}  on $A$ in
the sense of  \cite{BNpac}, this correspondence is essentially Corollary 4.2 there.
Similarly,  for the same reasons we will not discuss  class (3) in this paper, except 
in the final theorem.   By \cite[Theorem 3.3]{Sharma},  class (3) 
 is essentially the same as  `real operator algebra structure', 
that is $A$ is just the complexification of a real operator algebra $D = \{ x \in B : x = \bar{x} \}$, and we may rewrite $\bar{x} = 
a - i b$ if $x = a +  ib$ for $a, b \in D$.   Thus the variant of the main aspects of our paper in case (3) seem best treated within
the theory of real operator algebras.  
However it is worth saying that the theory
in our paper in case (3) may be viewed
as a transliteration of a chapter in the theory of real operator algebras.
  We also remark that if $A$ is unital or approximately unital then
 one can easily show using the Banach-Stone theorem for operator algebras 
(see e.g.\ \cite[Theorem 4.5.13]{BLM}), and the `opposite and adjoint algebras' discussed at the start of Section 2, 
that the matrix norm equality in (3) and (4) (resp.\ (1) and (2))  plus a `unital condition' force 
the involution to be multiplicative (resp.\  anti-multiplicative).

  If $A$ is a $C^*$-algebra then classes (1) and (4) are essentially the same after
applying the $C^*$-algebra involution $*$.
(Note that in this case the matrix norm equality in (1) or (4)   follows from the 
same equality for $1 \times 1$ matrices, that is that the involution is isometric.   Indeed it is well known
that $*$-isomorphisms of $C^*$-algebras are completely isometric.) 
Similarly classes (2) and (3) essentially coincide if $A$ is a $C^*$-algebra.
  
  We will mostly focus on class (1) for specificity.   In fact  
most of the results in the present paper apply
to all four classes, however it would be too tedious to state several cases of each result.   Instead we leave it to the 
reader to state the matching results in cases (2)--(4).   For example to get from case (1) to case (2) of results below
one replaces $a^\dagger$ by $a^\theta$, and $\dagger$-{\em selfadjoint elements}, that is elements satisfying $a^\dagger = a$, by 
elements with  $a^\theta = a$.   
We remark that if $A$ is an operator algebra with linear involution $\theta$, then $\{ a \in A : a = a^\theta  \}$ is a Jordan operator 
algebra in the sense of \cite{BWj}.    (We remark that  these `$\theta$-{\em selfadjoint} elements' need not generate $A$ as an operator 
algebra, unlike 
for involutions of type (1).)  
Most of our discussion of class (2) involves finding interesting examples of such involutions.   Indeed although classes (1)--(4) 
have similar theory from the viewpoint of our paper, the {\em examples} of algebras in these  classes are quite different in general.

Because of the ubiquity of the asterisk symbol
in our area of study, we 
usually write the involution on an operator $*$-algebra as $\dagger$,
and refer to, for example, $\dagger$-selfadjoint elements
or subalgebras, and $\dagger$-homomorphisms (the natural
morphisms for $*$-algebras).

A little more background and notation:   
A {\em unital} operator algebra has an identity of norm $1$,
and an {\em approximately unital} operator algebra has a
contractive approximate identity (cai). 
For background  on operator spaces
and operator algebras from an operator space 
point of view we refer the reader
to \cite{BLM,Pisbk}.
 Meyer's theorem  states that any operator algebra $A$ has a
unitization $A^1$ that is unique up to completely isometric isomorphism \cite[Corollary 2.1.15]{BLM}.  If $A$ is
nonunital then $A$ is of codimension 1 in the unital operator algebra $A^1$; otherwise set $A^1 = A$.
In this paper,  all projections $p \in A$ are orthogonal projections. If $X$ and $Y$ 
are sets then we write $XY$ for the {\em norm closure} of the span of terms of the form $xy,$ for $x\in X, y\in Y.$ The second dual $A^{\ast\ast}$ of an operator algebra $A$ is again an operator algebra,
which is unital if $A$ is approximately unital.

We recall that a  $C^{\ast}$-cover $(B, j)$ of an operator algebra $A$ is a $C^*$-algebra $B$ and a completely isometric
homomorphism $j : A \to B$ such that $j(A)$ generates $B$ as a $C^*$-algebra.   Sometimes we simply call this a $C^*$-algebra
generated by $A$.  There is a `biggest' and `smallest' $C^{\ast}$-cover,
$C^*_{\rm max}(A)$ and $C^*_e(A)$ (see \cite[Propositions 4.3.5 and 2.4.2]{BLM}). 
For example $C^*_{\rm max}(A)$ has the universal
property that any completely contractive representation $\pi : A \to B(H)$
extends to a $*$-representation of $C^*_{\rm max}(A)$ on $H$. Any
completely isometric homomorphism $j : A \to B$ into a $C^*$-cover $B$ of
$A$ generated by the copy of $A$, gives rise to a $*$-homomorphism $B \to C^*_e(A)$ which is
`the identity' on the copy of $A$.    

Because of the uniqueness of unitization, for an operator algebra $A$
we can define unambiguously ${\mathfrak F}_A = \{ a \in A : \Vert 1 - a \Vert \leq 1 \}$.  Then
 $\frac{1}{2} {\mathfrak F}_A = \{ a \in A : \Vert 1 - 2 a \Vert \leq 1 \} \subset {\rm Ball}(A)$.
Here and throughout ${\rm Ball}(X)$ denotes $\{ x \in X : \| x \| \leq 1 \}$.
  Similarly, ${\mathfrak r}_A$,
the {\em real positive} or {\em accretive} elements in $A$, is 
$\{ a \in A : a + a^* \geq 0 \}$, where the 
adjoint $a^*$ is taken in any $C^*$-cover of $A$.   
We write oa$(x)$ for the operator algebra generated by an operator $x$.  
By a {\em symmetry} we mean either a selfadjoint unitary operator, or a period 2 $*$-automorphism 
of a $C^*$-algebra, depending on the context.

 Recall that a projection in $A^{\ast\ast}$ is {\em open} in $A^{**},$ or $A$-open for short, if $p\in(pA^{\ast\ast}p\, \cap A)^{\perp\perp}$. That is, if  and only if there is a net $(x_t)$ in $A$ with 
$$x_t=px_t=x_tp=px_tp\to p,\, {\rm weak^*}.$$ 
This is a generalization of Akemann's notion of open projections for $C^*$-algebras. If $p$ is open in $A^{**}$ then  clearly
$$D=pA^{\ast\ast}p\cap A=\{a\in A: a=ap=pa=pap\}$$ 
is a closed subalgebra of $A,$ and the subalgebra $D^{\perp\perp}$ of $A^{\ast\ast}$ has identity $p.$ By \cite[Proposition 2.5.8]{BLM} $D$ has a cai. If $A$ is also approximately unital then a projection $p$ in $A^{\ast\ast}$ is closed if $p^{\perp}$ is open.

We call such a subalgebra $D$ a {\em hereditary subalgebra} of $A$ (or HSA) and we say that  $p$ is the {\em support projection} of the HSA $pA^{\ast\ast}p\cap A.$. It follows from the above that the support projection of a HSA is the weak* limit of any cai from the HSA.    The above correspondence between hereditary subalgebras and open projections is bijective and order preserving.  

\section{Involutive operator algebras}

We recall for any operator space $X$ the {\em opposite} and {\em adjoint} operator spaces $X^\circ$ and $X^\star$ from e.g.\ 1.2.25 in \cite{BLM}.  
Here $X^\circ$ is $X$ but with `transposed matrix norms' $|||[a_{ij} ]||| = \| [a_{ji} ] \|$.
Similarly $X^\star$ is the set of formal symbols $x^\star$ for $x \in X$, but with  the same operator space structure 
as  $\{ x^* \in B : x \in X \}$, if $X$ is (completely isometrically) a subspace of a 
$C^*$-algebra $B$.  
 
If $A$ is an operator algebra we write $A^1$ for the unitization of $A$.
If $X$ is an operator space then $I(X)$ and ${\mathcal T}(X)$ are respectively
the {\em  injective and ternary envelopes}  of $X$ from e.g.\ \cite[Chapter 4]{BLM} (the word `triple' is used in place 
of `ternary' in \cite{BLM}).   We will not apply the injective and ternary envelope in this paper, but 
include them in the following result since they will be important in future work on involutive operator algebras and spaces.   We recall from e.g.\
the last reference that
 $I(X)$ is an injective operator space containing $X$, and it is the unique such
 with the  `rigidity' (or with the `essential') property of 4.2.3 in \cite{BLM}.   Then ${\mathcal T}(X)$, sometimes called the {\em noncommutative Shilov
boundary} of $X$, is the smallest closed subspace of $I(X)$ containing $X$ which is closed under the ternary product $x y^* z$.   It may also be characterized by 
the universal  property in \cite[Theorem 8.3.9]{BLM}; and if $X$ is unital or is an approximately
unital operator algebra then ${\mathcal T}(X)$ is the $C^*$-envelope $C^*_e(X)$ mentioned earlier.

\begin{proposition} \label{ciop}  If $X$ is an operator space and $A$ is an operator algebra then
\begin{itemize}
\item [(1)]  $(A^1)^\circ = (A^\circ)^1$, $I(X^\circ) = I(X)^\circ, {\mathcal T}(X^\circ) =  {\mathcal T}(X)^\circ, C^*_e(A^\circ) = C^*_e(A)^\circ,$ and $C^*_{\rm max}(A^\circ) = C^*_{\rm max}(A)^\circ$.
\item [(2)] $(A^1)^\star= (A^\star)^1$, $I(X^\star) = I(X)^\star, {\mathcal T}(X^\star) =  {\mathcal T}(X)^\star, C^*_e(A^\star) = C^*_e(A)^\star,$ and $C^*_{\rm max}(A^\star) = C^*_{\rm max}(A)^\star$.
\end{itemize}
\end{proposition}

\begin{proof}       Note that  $(A^\circ)^1$ is a unital operator algebra containing $A^\circ$ as a codimension 1 subalgebra, so by
Meyer's theorem \cite[Corollary 2.1.15]{BLM} it must be the unitization.  Similarly for $(A^1)^{\star}$.
 The rest all follow by the universal properties defining these objects, and a diagram chase applying $\circ$ or $\star$ to the 
maps in the diagrams.  For example, such a strategy shows that $I(X)^\circ$ is injective.   It  contains $X^\circ$, and a similar 
strategy shows that it has the `rigidity' property (or `essential' property) characterizing the injective envelope.  
\end{proof}

{\bf Remark.}   There is a similar result for $\bar{X}$ and $\bar{A}$, which would be useful in treating class (3) from the start of Section
{\rm \ref{inv}}.   Here $\bar{X} = (X^{\star})^\circ$, and from this formula the proof of the result in this case is
clear.

\medskip

The following result, the involutive variant of  Meyer's theorem  \cite[Corollary 2.1.15]{BLM}, 
 is useful in treating involutions on operator algebras with no identity or approximate identity.

\begin{lemma} \label{mey}    Let  $A$ be a nonunital operator algebra with an involution of one of the  types {\rm (1)--(4)} at the start of Section
{\rm \ref{inv}}.
Then the  involution on $A$ has a unique extension to an involution of the same type on the unitization $A^1$, with the involution of $1$ 
being $1$.
\end{lemma}

\begin{proof}   For  operator $*$-algebras this is \cite[Lemma 1.15]{BKM}.   If $\theta : A \to A$ is a linear involution (type (2) in the 
list at the start of Section \ref{inv}), then by Meyer's theorem $a \mapsto \theta(a)^\circ$ extends to a unital completely isometric 
homomorphism  $A^1 \to (A^\circ)^1$.   Composing this with $\circ$, and using the fact that  $(A^1)^\circ = (A^\circ)^1$
from Proposition \ref{ciop}, we obtain our result.   The other cases are similar or easier.
\end{proof}

\begin{proposition}\label{dmeyer1}
Let $A$ and $B$ be operator algebras  with $A$ nonunital. Also suppose that there exists involutions on $A$  and $B$ 
of one of the  types {\rm (1)--(4)} at  the start of Section {\rm \ref{inv}}.  Let $\pi: A\to B$ be a completely contractive (resp.\ completely isometric) involution 
preserving homomorphism.  Then there is a unital completely contractive (resp.\ completely isometric) involution 
preserving homomorphism extending $\pi:$ from $A^1$ to $B^1$ (for the completely isometric case we also need $B$ nonunital).
\end{proposition}
\begin{proof}
By Lemma \ref{mey}, we know that  both $A^1$ and $B^1$ are operator algebras with
 the same type of involution.  The unital extension of $\pi$ to $A^1$ is completely contractive (resp.\ completely isometric) by \cite[Theorem 2.1.13 and Corollary 2.1.15]{BLM}.
It is easy to check that it is also involution preserving. 
\end{proof}

{\bf Remark.}
One may replace completely contractive (resp.\ completely isometric) by contractive (resp.\ isometric) in the last  result.
Thus the unitization $A^1$ of an involutive operator algebra is unique up to (completely) isometric involutive isomorphism.	

\medskip

We now characterize class (2) from the start of Section
{\rm \ref{inv}}.   A {\em conjugation} on a complex Hilbert space $H$ is a
conjugate linear period $2$ isometry $u : H \to H$.   If $j : H \to \bar{H}$ is the canonical conjugate linear map into the 
conjugate Hilbert space 
then $ju$ is unitary, so $\langle uy , ux \rangle = \langle ju x, ju y \rangle_{\bar{H}} = \langle x, y \rangle$ for $x, y \in H$.  
An operator $T$ on $H$ is called $c$-{\em symmetric} if $c T c = T^*$, and  is called {\em complex symmetric} if 
it is $c$-symmetric for a conjugation $c$ on $H$.    The class of complex symmetric operators is very large and significant
(see e.g.\ \cite{GP}).  

\begin{theorem} \label{lininvch} Let  $A$ be an operator algebra.  The following are equivalent:
\begin{itemize}
\item [(i)] $A$ is an operator algebra with linear involution (that is,  type {\rm (2)} in the 
list at the start of Section  {\rm \ref{inv}}) $\theta$.
\item [(ii)]  There exists a $C^*$-algebra $B$ generated by  $A$ 
(or by $A^1$), and a period 2 $*$-anti-isomorphism  $\rho : B \to B$ with $\rho(A) = A$.
\item [(iii)] There exists a conjugation $c$ on a complex Hilbert space $H$ on which $A$ may be completely isometrically represented 
as an operator algebra such that $c A^* c \subset A$ (here we are identifying $A$ with its image in $B(H)$).
\end{itemize}
We may take $\theta$ in {\rm (i)} to be the restriction to $A$ of the $\rho$ in {\rm  (ii)}, or of the map $T \mapsto c T^* c$ in {\rm (iii)}.
We may take $B$ in {\rm (ii)} to be $C^*_e(A)$ (or $C^*_e(A^1)$), or $C^*_{\rm max}(A)$ (or $C^*_{\rm max}(A^1)$).
\end{theorem}

\begin{proof}    (i) $\Rightarrow$ (ii) \  We may assume that $A$ is unital by  Proposition \ref{dmeyer1}.
Given (i), the map $A \to A^\circ : a \mapsto \theta(a)^\circ$ is a completely isometric isomorphism, so extends to a 
$*$-isomorphism $B = C^*_e(A) \to C^*_e(A^\circ)$.  The latter algebra
equals $C^*_e(A)^\circ$ by Proposition \ref{ciop}.   This gives a  $*$-anti-isomorphism on $B$
taking $A$ onto $A$, which is easily checked to be period 2.    Similarly with 
$B = C^*_{\rm max}(A)$  
using the universal property of these $C^*$-algebras 
and  the appropriate 
item in Proposition \ref{ciop}.

(ii) $\Rightarrow$ (iii) \ The map $\pi(b) =  \rho(b)^*$ is a  period 2 conjugate linear $*$-isomorphism on $B$.  
Then $D = \{ b \in B : b =  \rho(b)^* \}$ is a real $C^*$-algebra, which we can represent on a real Hilbert space $K$.   Let $H = K_c$ 
and $c : H \to H$ be the canonical conjugation.   Then $B$ may be viewed as the $C^*$-algebra $D + i D$ acting on $H$ by
$(a + ib)(\xi + i \eta) = a \xi - b \eta + i(a \eta + b \xi)$.   And $\pi(a+ib) = a - ib$ by definition of $D$.   Then $$c (a + ib) c
(\xi + i \eta) =  c( a \xi + b \eta + i(- a \eta + b \xi)) = a \xi + b \eta + i (a \eta - b \xi) = \pi(a+ib) (\xi + i \eta).$$
So $\pi(b) = c b c$, and $\rho(b) = c b^* c$.

 (iii) $\Rightarrow$ (i) \ The map $\theta(a) =  c a^* c$ is a linear period 2 anti-isomorphism from $A$ to $A$, and  $$\| [ a_{ij} ] \|  = \| [ a_{ji}^* ] \|  =
\| [ c a_{ji}^* c ] \|=  \| [ a_{ji}^\theta ] \| $$ 
for $[a_{ij}] \in M_n(A)$.   \end{proof}

A {\em symmetric  operator algebra} is an operator algebra $A$ with $A = A^\circ$ completely 
isometrically via the identity map.  These were introduced in
\cite{BComm} where it was observed that such algebras were commutative, etc.  They are characterized by the following result, 
which says that such algebras must be operator
algebras of matrices that equal their transpose
(i.e.\ which are symmetric with respect to the transpose as matrices).   By an operator algebra of matrices
we mean a subalgebra of $\Mdb_I$ for a cardinal $I$,
where $\Mdb_I = B(\ell^2_I)$ thought of as $I \times I$ matrices.
Any operator algebra $A$ of symmetric matrices is  commutative, since $ab = (ab)^T = ba$ for $a, b \in A$.   Hence $A$ is 
 the algebra
generated by a set of commuting symmetric matrices in $\Mdb_I$. 
  
\begin{corollary} \label{chsym}  An  operator algebra $A$  is symmetric if and only if there exists a conjugation $c$ on a complex Hilbert space $H$ on which $A$ may be completely isometrically represented 
as an operator algebra,  such that $c a c = a^*$ for all $a \in A$ (here we are identifying $A$ with its image in $B(H)$).
That is, if and only if  there exists a  completely isometric representation of $A$ as $c$-symmetric operators
for a  conjugation $c$.    Equivalently, $A$ is symmetric if and only if there exists an orthonormal basis  ${\mathfrak B}$ for 
a complex Hilbert space $H$ on which $A$ has been completely isometrically represented 
as an operator algebra,  such that the matrix with respect to  ${\mathfrak B}$ of every element in $A$ equals its transpose.
\end{corollary} 

\begin{proof}    Set $\theta$ equal the identity map and apply the last theorem and its proof  to see that 
$A$  is symmetric if and only if the desired  conjugation $c$ on $H$ exists with $c a^* c = a$ for $a \in A$.  
From this the first two if and only if's follow. 
The last statement follows from the fact alluded to earlier that the  $c$-symmetric operators are precisely the 
operators whose matrix with respect to  some appropriate orthonormal basis  ${\mathfrak B}$ is symmetric.  This follows from
the well known trick  (certainly well known 
in the theory of complex symmetric operators) that if 
$K = \{  \xi = c \xi \in H \}$, then $K^\perp = (0)$, so  a (real)  orthonormal basis ${\mathfrak B}$ for $K$ is  a (complex) 
basis for $H$.  Also $c T c = T^*$ implies the matrix with respect to  ${\mathfrak B}$ is symmetric.  
\end{proof}  

\begin{corollary} \label{chsym2}  An  operator algebra $A$  is commutative
if and only if it is isometrically isomorphic to an operator algebra of matrices 
that equal their transpose.
\end{corollary}

\begin{proof}  If $A$  is commutative then it is 
isometrically isomorphic to $\{ (a, a^\circ) \in A \oplus A^{\circ} : a \in A \}$,
which is a symmetric operator algebra \cite{BComm}.
 The rest follows from Corollary \ref{chsym} (and the observation above that an algebra of symmetric matrices is  commutative).
\end{proof}

\begin{corollary} \label{chsym3}  The algebra generated by any operator on a Hilbert space is
 isometrically isomorphic to the algebra generated by a complex symmetric operator 
on another Hilbert space.
\end{corollary}

\begin{proof}   As in the proof of Corollary \ref{chsym2}, this algebra is isometrically isomorphic to
a symmetric operator algebra, and by  Corollary \ref{chsym}
 its generator may be taken to be  
$c$-symmetric for a conjugation $c$. \end{proof}

There are similar characterizations for the other three classes of `involutions' considered  at the start of Section
{\rm \ref{inv}}.
Indeed the result matching Theorem \ref{lininvch} for operator $*$-algebras 
is the following, mostly from \cite[Section 1]{BKM}.

\begin{theorem} \label{opinvch} Let  $A$ be an operator algebra.  The following are equivalent:
\begin{itemize}
\item [(i)] $A$ is an operator $*$-algebra.
\item [(ii)]  There exists a $C^*$-algebra $B$ generated by  $A$
(or of $A^1$), and a period 2 $*$-automorphism  $\rho : B \to B$ with $\rho(A) = A^*$.  \item [(iii)] There exists a symmetry
 $u$ on a complex Hilbert space $H$ on which $A$ may be completely isometrically represented
as an operator algebra such that $u A^* u \subset A$ (here we are identifying $A$ with its image in $B(H)$).
\end{itemize}
We may take $\theta$ in {\rm (i)} to be the restriction to $A$ of $\rho(\cdot)^*$
for $\rho$ as in {\rm  (ii)}, or to be the map $T \mapsto u T^* u$ in {\rm (iii)}.
We may take $B$ in {\rm (ii)} to be $C^*_e(A)$ (or $C^*_e(A^1)$), or $C^*_{\rm max}(A)$ (or $C^*_{\rm max}(A^1)$).
\end{theorem}

\begin{proof} This is proved
in \cite[Section 1]{BKM}, except for the
assertion about $C^*_{\rm max}(A)$.  If $\rho : A \to C^*_{\rm max}(A)$
is the canonical `inclusion' let $\pi : A\to C^{\ast}_{\max}(A)$ be the completely isometric homomorphism defined by $\pi(a)= \rho(a^{\dagger})^{\ast}.$ By the universal property of $C^{\ast}_{\max}(A),$ there exists a unique $\ast$-homomorphism $\sigma: C^{\ast}_{\max}(A)\to C^{\ast}_{\max}(A)$ such that $\sigma(\rho(a))=\pi(a)=\rho(a^{\dagger})^{\ast}$ for any $a\in A.$ 
Moreover, $\sigma$ has order 2 since
$$\sigma^2(\rho(a))=\sigma(\rho(a^{\dagger})^{\ast})=\sigma(\rho(a^{\dagger}))^{\ast}=\rho(a),$$
and since $\rho(A)$ generates $C^{\ast}_{\max}(A)$ as a $C^{\ast}$-algebra.  The final assertion
follows by extending to the unitization and using  $C^{\ast}_{\max}(A^1) = C^{\ast}_{\max}(A)^1$ from Proposition \ref{ciop}.
 \end{proof}  

It is natural to ask if in item (ii) in Theorems \ref{lininvch} or 
\ref{opinvch} (ii), or in matching results 
for the other types of involutions,
one may use {\em any} $C^*$-algebra generated by a completely isometric 
copy of $A$.   The answer is in the negative, as one sees in the following result and the example following it.

\begin{definition} \label{ivltcmptb}   Suppose that an operator algebra $A$ has an involution $\nu$ of 
one of the types {\rm (1)--(4)}  at the start of Section \ref{inv}.  
If a $C^{\ast}$-cover $(B, j)$ of $A$ has an involution $\omega$ of the same type,
and if  $j(a)^{\omega}=j(a^{\nu}),$ for any $a\in A,$ we say that the involution on $B$ is {\em compatible with} $A.$ 
\end{definition}

\begin{lemma}
Suppose that $A$ is an operator algebra (possibly not approximately unital) with involution $\nu$ of 
 type {\rm (1)} (resp.\ type {\rm  (2)})  at the start of Section {\rm \ref{inv}}, 
 and $(B,j)$ is a $C^{\ast}$-cover of $A$. Then $B$ has an involution compatible with $A$ if and only if there exists an order $2$ $\ast$-automorphism 
(resp.\  $\ast$-antiautomorphism) $\sigma: B\to B$ such that $\sigma(j(a^{\nu}))=j(a)^{*}$ (resp.\ $\sigma(j(a^{\nu}))=j(a)$)  for any $a\in A$.
\end{lemma}
\begin{proof}
($\Rightarrow$) If $B$ has an involution $\omega$ compatible with $A,$ then  $j(a)^{\omega}=j(a^{\nu}),$ for all $a\in A.$ Define $\sigma: B\to B$ 	by $\sigma(b)=(b^{\ast})^{\omega}$ (resp.\ $b^\omega$) for any $b\in B.$ Then it is easy to see that $\sigma$ is an order $2$ $\ast$-automorphism
(resp.\  $\ast$-antiautomorphism).

($\Leftarrow$) The involution on $B$ is defined by  $b^{\omega}=\sigma(b)^{\ast}$ (resp.\ $b^{\omega}= \sigma(b)$)  for any $b\in B.$ Then $B$ is a $C^{\ast}$-algebra with involution which is compatible with $A.$
\end{proof}

\begin{example}    Let $A(\Ddb)$ be the disk algebra.  By the above $C^*_{\rm max}(A(\Ddb))$ and $C(\Tdb) = C^*_e(A(\Ddb))$
 are $C^*$-covers of $A(\Ddb)$ which are compatible
 with the involution $\overline{f(\bar{z})}$ on $A(\Ddb)$.   So too is $C(\overline{\Ddb})$ clearly, with the same involution $\overline{f(\bar{z})}$.
The Toeplitz $C^*$-algebra is a well known $C^*$-cover of the disk algebra $A(\Ddb)$, however we show that it is not compatible 
with the involution $\overline{f(\bar{z})}$ on $A(\Ddb)$.  
Let $S$ be the unilateral shift on $l^2(\Ndb_0)$ and $\oa{S}$ be the operator algebra generated by $S.$ Then $\oa{S}$ is an operator $\ast$-algebra with trivial involution induced by $S^{\dagger}=S.$ Suppose that the Toeplitz $C^*$-algebra $C^{\ast}(S)$ has an involution compatible with $\oa{S}.$ Then there exists an order-2 $\ast$-isomorphism $C^{\ast}(S)$ such that $\sigma(S^{\dagger})=S^{\ast}.$ Moreover, we have
$$I=\sigma(I)=\sigma(S^{\ast}S)=\sigma(S)^{\ast}\sigma(S)=SS^{\ast}\neq I,$$ 
which is a contradiction.

It is similarly not hard to find (using \ref{Shb} below)
 commutative $C^*$-algebras $C(K)$ generated by a function algebra $A$ with linear involution $\theta$ (such as $A = A(\Ddb)$), such that $\theta$ does not extend to a  linear involution on $C(K)$. 
\end{example}

\begin{lemma} \label{bidu}    Let  $A$ be an operator algebra with an involution of one of the  types {\rm (1)--(4)} at the start of Section
{\rm  \ref{inv}}.
Then the  involution on $A$ has a unique extension to a weak*
continuous involution of the same type on the bidual $A^{**}$.
\end{lemma}

\begin{proof}  We will just prove this in the case of a linear involution
$\theta$; the others are similar.  The associated completely isometric
homomorphism  
$A \to A^\circ : a \mapsto \theta(a)^\circ$
extends to a weak* continuous 
completely isometric
homomorphism
$A^{**} \to (A^\circ)^{**}$.  However it is an
easy exercise
to see that $(A^\circ)^{**} \cong (A^{**})^\circ$.
Composing with $\circ$ we obtain a 
 weak* 
continuous  linear involution on $A^{**}$ extending
$\theta$.     \end{proof}

We mention that the Cayley transform $\kappa$ 
and the ${\mathfrak F}$ transform of \cite[Section 2.2]{BRord},  important tools in 
the area of the later sections of our paper, do work well with respect to involutions.  
For example suppose that  $\dagger$ is the involution 
on an operator $*$-algebra, and $\sigma$ the associated
$*$-automorphism on a (compatible) $C^*$-cover.
If $x$ is real positive then $\sigma(x^*) = x^\dagger$ is real positive,
and   $$\sigma \kappa(x^\dagger) = \sigma((x^\dagger - 1)( x^\dagger + 1)^{-1})
= (x^*  - 1)( x^* + 1)^{-1} = \kappa(x)^*.$$  So $\kappa(x^\dagger) = \kappa(x)^\dagger$.
Similarly, if $x$ is a contraction with $1-x$ invertible then the same is true for $x^\dagger$ and the inverse Cayley transform
$\kappa^{-1}(x^\dagger) = (1+x^\dagger)(1- x^\dagger)^{-1}$ is real positive, and  must equal
$\kappa^{-1}(x)^\dagger$.   The ${\mathfrak F}$-transform is
 ${\mathfrak F}(x) = \frac{1}{2}(1 + \kappa(x)) = x(1+x)^{-1}$. 
Following the proof in Lemma 2.5 in \cite{BRord}, it is easy to see
that for any operator $\ast$-algebra $A,$ the ${\mathfrak F}$-transform maps
the $\dagger$-selfadjoint elements in ${\mathfrak r}_A$ 
bijectively onto the set of
$\dagger$-selfadjoint elements in $\frac{1}{2}{\mathfrak F}_A$ of norm $<1.$

\section{Examples} \label{Exs}

  We give many  examples of operator $*$-algebras
and operator algebras with linear involution here.   Of course any
 real operator algebra  at all gives an example of the third type of involution mentioned at the start of Section \ref{inv}, namely
the complexification.  We will not consider
these here.   

\subsection{Examples from noncommutative differential
geometry}

Several examples of operator $*$-algebras were given in \cite{BKM}, most of them examples from noncommutative differential
geometry (historically the first such example being due to Mesland). 
Other examples from noncommutative differential
geometry may be found in other recent papers of Kaad, Mesland, and their coauthors.

\subsection{Function algebra examples}   \label{Shb}
Let $A$ be a uniform algebra  (with minimal operator space structure, see 1.2.21
 in \cite{BLM}).  Then $A\subset C^{\ast}_e(A)=C(\partial A),$ where $\partial A$ is the Shilov boundary of $A$
(see e.g.\ \cite[Section 4.1]{BLM}). If $A$ is an operator $*$-algebra (resp.\  has linear involution $\theta$), then there exists 
a period $2$ homeomorphism $\tau : \partial A\to \partial A$ such that $f^{\dagger}(\omega)=\overline{f(\tau(\omega))}$  (resp.\ $f^\theta = f \circ \tau$) for any $f \in A$.     From this formula it is easy to write down function algebra examples.  
For example, the disk algebra $A(\Ddb)$ is an operator $*$-algebra with $f^\dagger(z) = \overline{f(\bar{z})}$,
and so are its  closed $\dagger$-ideals of functions e.g.\ vanishing at 0, or at 1.  The latter ideal  is
interesting from the perspective of approximate identities:  it is nonunital, is a 
$\dagger$-ideal, and  has a real positive $\dagger$-selfadjoint cai (see Lemma \ref{dcai} below), etc).
Similarly $H^\infty(\Ddb)$ is a dual  operator $*$-algebra 
 with the same involution.   This involution is
weak* continuous, and extends to an involution on the von Neumann algebra $L^\infty(\Tdb)$.

These two algebras also have linear involution $f^\theta(z) = f(-z)$.   This and the identity map are the only linear involutions
on $A(\Ddb)$ and $H^\infty$, by the well known theory of automorphisms of these algebras.  

 We recall that a {\em Q-algebra} is an operator algebra quotient of 
a function algebra  (with minimal operator space structure)  by a closed ideal.   Q-algebras are symmetric  operator algebras, and in particular have a linear involution.  If the function algebra has an involution making it an operator $*$-algebra, and the ideal
is involutive, then we call the quotient an {\em involutive Q-algebra}.    We will see later that for example the algebra generated by the 
Volterra operator is an involutive Q-algebra.   

\subsection{Examples from complex symmetric and $*$-exchangeable operators} 

An operator $T$ in a $C^*$-algebra $B$ will be called $*$-{\em exchangeable} if $\| p(T, T^*) \| = \| p(T^*, T) \|$
for any polynomial $p$ in two noncommuting variables.  One may use polynomials without constant term here if one wishes.
Indeed if the equality holds for such polynomials $p$ then it follows that the map $p(T, T^*) \mapsto  p(T^*, T)$  is well defined 
on a dense subset of $C^*(T)$, hence extends to a
$*$-homomorphism $\sigma$ on $C^*(T)$ taking $T$ to $T^*$, and extends
further to $C^*(1,T)$.  This shows that 
the norm equality holds for polynomials with constant terms too.  
It is easy to see that $\sigma$ is a period 2 $*$-automorphism.  

For a  polynomial $p$ of one variable we write  $p^\dagger$ for the same polynomial but with coefficients replaced by their complex conjugate
(that is, $p^\dagger(z) = \overline{p(\bar{z})}$).  

\begin{theorem} \label{oneg} Let $A$ be an operator algebra with a single generator $T$.   The following are equivalent:
\begin{itemize}
\item [(i)] For $n \in \Ndb$ and polynomials $p_{ij}$ for $1 \leq i, j \leq n$
we have $$\| [ p_{ji}^\dagger (T) ] \|
= \| [ p_{ij}(T) ] \| .$$
\item [(ii)]  $T$ is $*$-exchangeable in some $C^*$-algebra generated by 
(a completely isometrically homomorphic copy of) $A$.
\item [(iii)] $A$ is an operator $*$-algebra with $T$ $\dagger$-selfadjoint.
\item [(iv)] There exists a symmetry $u$ on a Hilbert space $H$ on which $A$ may be completely isometrically represented 
as an operator algebra such that $T^* = u T u$ (here we are identifying $T$ with its image in $B(H)$).
\end{itemize}
In 
 {\em (i)} one may if one wishes use only polynomials with no constant term.
\end{theorem}

\begin{proof} 
(iv)  $\Rightarrow$ (i) \ Given 
linear symmetry $u : H \to H$ with $T^* = uTu$
then $p(T)^* =  p^\dagger (T^*) = u p^\dagger (T) u$,
so that
$$\| [ p_{ji}^\dagger (T) ] \|
= \| [  u \, p_{ji}(T)^* \, u] \| = \| [ p_{ij}(T) ] \| .$$

(i) $\Rightarrow$ (iii) \ If  for polynomials $p_{ij}$ with no constant term 
we have $\| [ p_{ij}(T) ] \| = \| [ p_{ji}^\dagger (T) ]  \| = 
\| [ p_{ij} (T^*) ]  \|$, then the map $p(T) \mapsto p(T^*)$ is well defined and completely isometric.  Here $*$ is the involution on a $C^*$-algebra
containing $A$.  It extends to a completely isometric surjective homomorphism oa$(T) \to {\rm oa}(T^*)$.
Composing this with the involution $*$ we obtain an involution on oa$(T)$ making it an 
operator $*$-algebra with $T$ $\dagger$-selfadjoint.   
(If one wishes then we can extend the involution to the unitization by Proposition \ref{dmeyer1}, which implies the equality in (i) for polynomials
with constant term.)

(iii) $\Rightarrow$ (ii) \  If oa$(T)$ has such involution then by the characterization of operator $*$-algebras in Theorem \ref{opinvch} 
 there exists a $*$-isomorphism  $C^*_e({\rm oa}(T)) \to C^*_e({\rm oa}(T))$
 taking   $T^\dagger = T$ to $T^*$.    
Equivalently (as in the discussion above the theorem), $p(T, T^*) \mapsto p(T^*, T)$ is a well-defined  isometry.  

(ii) $\Rightarrow$ (iii) \ If $T$ is $*$-exchangeable in some $C^*$-algebra $B$ generated by $A$, then as explained above the theorem
we have a period 2 
$*$-automorphism $\sigma : B \to B$ with $\sigma(T) = T^*$.   The restriction of $\sigma$ to $A$ maps onto $A^*$.  So $A$ is an operator $*$-algebra with $T$ $\dagger$-selfadjoint if we 
define $a^\dagger = \sigma(a)^*$.

(iii) $\Rightarrow$ (iv) \ By the characterization of operator $*$-algebras in Theorem \ref{opinvch}, there exists a symmetry $u$ on a Hilbert space $H$ on which $A$ may be completely isometrically represented 
as an operator algebra such that $a^\dagger = u a^* u$ for all $a \in A$.  Setting $a = T$ we obtain  $T^* = u T u$. 
\end{proof}

There is a similar result for operator algebras with linear involution $\theta$ with 
$\theta(T) = T$.   The analogue of condition  (ii) in Theorem \ref{oneg} is the condition called $g$-normality in \cite{GuoJiZhu}, namely that 
$\| p(T, T^*) \| = \| p^\dagger (T^*, T) \|$
for any polynomial $p$ in two free variables.   Here $p^\dagger$ is obtained from $p$ by conjugating each coefficient.
The equivalence of (ii) and (iv) is known: after our paper was written we found this equivalence  in \cite{SZ} with a quite different proof.
The paper \cite{ZZ} also
contains some other very interesting related results.

\begin{theorem} \label{cosym} Let $A$ be an operator algebra with a single generator $T$.   The following are equivalent:
\begin{itemize}
\item [(i)] For $n \in \Ndb$ and polynomials $p_{ij}$ for $1 \leq i, j \leq n$
we have $$\| [ p_{ji} (T) ] \|
= \| [ p_{ij}(T) ] \| .$$
\item [(ii)]  $T$ is $g$-normal in some $C^*$-algebra generated by $A$.
\item [(iii)] $A$ is a symmetric  operator algebra (that is $I_A$ is
a
 linear involution). 
\item [(iv)]   There exists  a Hilbert space $H$ on which $A$ may be completely isometrically represented 
as an operator algebra such that $T$ becomes a complex symmetry on $H$ (in the sense defined above {\rm 
Theorem \ref{lininvch}}). 
\end{itemize}
In {\rm (i)} one may if one wishes use only polynomials with no constant term.
\end{theorem}

\begin{proof} (iv)  $\Rightarrow$ (i) \ Given 
 conjugation $c : H \to H$ with $T^* = cTc$
then $p(T)^* =  p^\dagger (T^*) = c p (T) c$,
so that
$$\| [ p_{ij} (T) ] \| = \| [ p_{ji}(T)^* ] \| 
= \| [  c \, p_{ji}(T) \, c] \| = \| [ p_{ji}(T) ] \| .$$

(i) $\Leftrightarrow$ (iii) \ Obvious.  

(iii) $\Rightarrow$ (ii) \  If oa$(T)$ has such involution then by Theorem 
\ref{lininvch}  there exists a $*$-antiautomorphism  $C^*_e({\rm oa}(I,T)) \to C^*_e({\rm oa}(I,T))$
 taking   $T$ to $T$.    Composing with $*$, we get a conjugate linear $*$-automorphism of $C^*_e({\rm oa}(I,T))$ taking $T$ to $T^*$. 
Equivalently, $p(T, T^*) \mapsto p^\dagger(T^*, T)$ is a well-defined  isometry.  

(ii) $\Rightarrow$ (iii)  \ If $T$ is $g$-normal in some $C^*$-algebra $B$ generated by $A$, then we have a period 2 
$*$-antiautomorphism $\sigma : B \to B$ with $\sigma(T) = T$.   The restriction of $\sigma$ to $A$ maps onto $A$.  So $A$ is an operator algebra with
linear involution  $\theta$ with
$T^\theta = T$.

(iii) $\Rightarrow$ (iv) \ Immediate from  Corollary \ref{chsym}. 
\end{proof}

\begin{example}  One may ask if all operators $T$ satisfy the conditions in the last theorem, or in the one before it.  However Halmos' example $x = 2E_{12} + E_{23}$ in $M_3$
may be shown to be a counterexample.    Since $x^\intercal= x^*$ (we write $\intercal$ for the transpose) the same example will work for both.
Indeed one can show that $x$ generates $M_3$ as a $C^*$-algebra, and since this is simple we have  
$C^*_e({\rm oa}(x)) = M_3$.  Any $*$-automorphism of $M_3$ is inner, and also  $*$-antiautomorphisms of $M_3$ 
are of form $u^* a^\intercal u$ for a unitary $u \in M_3$.   An easy matrix computation show that there are no unitary solutions to
 $u^* x u = x^\intercal = x^*$.   Thus $x$ is not $*$-exchangeable or g-normal in $M_3$, hence oa$(x)$  is not symmetric
nor is an operator $*$-algebra with $x$ $\dagger$-selfadjoint.     On the other hand, the 
matrix $x \oplus x^\intercal$ in $M_6$ does satisfy the conditions in the last two theorems (this may be seen similarly to the  idea in the proof of
Corollary \ref{chsym2}).
 \end{example}

Many `truncated Toeplitz operators' are  complex symmetric, and some are $*$-exchangeable,  giving by 
the theorems above examples  of operator $*$-algebras,  and operator algebras with linear involution. 
To see these assertions it is helpful to recall the Sz.\ Nagy-Foias model theory for contractions \cite{NFK,Berc}.   
For many contractions $T$ it is known that $T$ is 
unitarily isomorphic to a truncated Toeplitz operator, a so-called
{\em Jordan block} \cite[Chapter 3]{Berc}, namely the compression 
$S(u) = P_K S_{|K}$ of the unilateral shift $S$, viewed as multiplication by $z$ on $H^2$, to the subspace $K = H^2 \ominus u H^2$, 
for a (nonconstant) inner function $u$ on the disk.  Thus  the weak* closed algebra  $A_T$ generated by $T$ (and $I$) is completely isometrically 
and weak* homeomorphically isomorphic to the weak* closed algebra generated by 
$S(u)$ (and $I$).
On the other hand, the last   weak* closed algebra  is known to be 
equal to the commutant $\{ S(u) \}'$, and is isometrically
weak* homeomorphic to the quotient $H^\infty/ u H^\infty$ (see 
\cite[Corollary 1.20]{Berc}). 

\begin{lemma} \label{symm} The weak* closed operator algebra generated by a Jordan block $S(u)$ is symmetric,
indeed is a Q-algebra, 
for every inner function $u$.  Thus $S(u)$ satisfies the conditions of the last theorem.
\end{lemma} 

\begin{proof}    In  \cite[Corollary 1.20]{Berc}, it is shown that $A_{S(u)} = \{ S(u) \}'$ is isometrically 
weak* homeomorphic to the quotient $H^\infty/ u H^\infty$.  
Following
the ideas in the proof of \cite[Corollary 1.20]{Berc} one can see that 
this isometry is a complete isometry.   The functional
calculus  $H^\infty \to A_{S(u)}$ for $S(u)$ is a complete contraction since it has a positive unital,
hence completely positive
and completely contractive, extension to $L^\infty$.
Thus we have an isometric complete contraction $H^\infty/ u H^\infty \to A_T$.
The unilateral shift $S$ is a minimal 
isometric dilation of $S(u)$.  Suppose that $x = [ x_{ij} ]
\in {\rm Ball}(M_n(A_{S(u)}))$.
Then $S^{(n)}$ is a minimal isometric dilation of $S(u)^{(n)}$.
By the commutant lifting  theorem 
(e.g.\ \cite[Theorem 1.10]{Berc}) there exists $y = [y_{ij} ] 
\in {\rm Ball}(M_n(B(H^2))$ such that $y \in \{ S^{(n)} \}'$,
so that $y_{ij} \in \{ S \}'$, and
$P_{K^{(n)}} y_{|K^{(n)}} = x$.   Thus $P_K (y_{ij})_{|K} = x_{ij}$.  
Since $y_{ij} \in \{ S \}'$, and the $H^\infty$ functional
calculus is a complete isometry $H^\infty \to \{ S \}'$,
we see that $y_{ij} = f_{ij}(S)$ for 
$[f_{ij}] \in {\rm Ball}(M_n(H^\infty))$.  Thus
$[f_{ij} + u H^\infty] \in {\rm Ball}(M_n(H^\infty/ u H^\infty))$
is a preimage of $x$.   It follows that 
$A_{S(u)} \cong H^\infty/ u H^\infty$ completely 
isometrically (and weak* homeomorphically).
Now 
 $H^\infty/ u H^\infty$ is a $Q$-algebra. 
 Hence $A_{S(u)}$ is a   $Q$-algebra and symmetric operator algebra. 
Its subalgebra oa$(S(u))$ is thus also symmetric, so $S(u)$  satisfies the conditions of the last theorem.  \end{proof} 

As a consequence,   the large class of contractions
$T$ unitarily equivalent to a  Jordan block $S(u)$ for 
some (nonconstant) inner function $u$ on the disk,
all generate symmetric operator algebras,
in particular operator algebras having  linear involution.

Turning to more specific examples, the  Volterra operator $Vf(x) = \int_0^x \, f(t) \, dt$ on $L^2([0,1])$  is both $*$-exchangeable and complex symmetric
(the latter via the conjugation $cf(t) = \overline{f(1-t)}$).  Thus the operator 
algebra  generated by $V$ is both an operator $*$-algebra and has linear involution.    The same is true for the  weak* closed algebra
generated by $V$.
These may be viewed as  infinite dimensional versions of the upper triangular matrices.     Indeed the Volterra operator $V$ is unitarily equivalent to $S(u)$ with $u(z) = \exp ((z+1)(z-1)^{-1})$, by e.g.\ \cite[Lemma 3.18 on p.\ 97]{Berc}, and this
$u$ is invariant under the involution $f^\dagger(z) = \overline{f(\bar{z})}$.   
Thus $H^\infty/ u H^\infty$ is an operator $*$-algebra, indeed is an  involutive Q-algebra, and also is 
a dual operator $*$-algebra.
  This is because $\dagger$ is a weak* continuous 
involution on $H^\infty$.  Hence $A_V$ is a dual operator $*$-algebra.  
Similarly, the norm closed algebra oa$(V)$ generated by the Volterra algebra 
is completely
isometrically isomorphic to $A_1(\Ddb)/u A_1(\Ddb)$
where $A_1(\Ddb)$ are the disk algebra functions vanishing
at $1$ (the isomorphism $H^\infty/ u H^\infty
\to A_V$ restricts to an isomorphism $A_1(\Ddb)/u A_1(\Ddb)
\to {\rm oa}(V)$, see \cite{PW}).
The latter quotient again is
an operator $*$-algebra (since $A_1(\Ddb)$ and its
ideal $u A_1(\Ddb)$
are invariant under the involution $\dagger$).
So oa$(V)$ is an  involutive Q-algebra, with $\dagger$-selfadjoint
generator.  The associated $\dagger$-selfadjoint 
contractive generator is $1 - (1+V)^{-1} = V(I + V)^{-1}$ (see the discussion  
just above Section \ref{Exs}),
which corresponds to the image of  $(1-z)/2 \in A_1(\Ddb)$.  It is known to be a 
radical Banach algebra \cite{PW} so the spectrum of every element 
is $(0)$.   Thus this is an example of
an operator $*$-algebra such that every $\dagger$-selfadjoint element has real spectrum,
but  which is not a $C^*$-algebra.     Indeed in this algebra for every $a \in A$ we 
have Sp$(a^\dagger a) \subset [0,\infty)$.  

Slightly more generally a contraction operator unitarily 
equivalent to Jordan block $S(\theta)$ 
for an inner function $\theta$, generates 
an  operator $*$-algebra with $T^\dagger = T$
if $\theta(\bar{z}) = \overline{\theta(z)}$.   Such inner functions include Blaschke products with real zeroes
and the function $u$ in the previous paragraph.

\subsection{Examples based on upper triangular matrices} The upper triangular $n$ by $n$ matrix algebra is an example which has all four types of involutions mentioned at the start of Section \ref{inv}.
 The $*$-algebra involution is given by 
$x^\dagger = u_n x^* u_n$
where $u_n$ is the order reversing $n \times n$ permutation 
matrix.   Similarly $ u_n x^\intercal u_n$  is a linear involution, where $\intercal$ is the transpose.   Similarly,   the infinite dimensional version of the 
 upper triangulars acting on  $l^2(\Zdb)$ is an operator $*$-algebra and 
operator algebra with linear involution.  Here  $u((a_n)) = (a_{-n})$ for $(a_n) \in 
l^2(\Zdb)$, and $x^\dagger = u x^* u$, etc.      These algebras have as  one `involutive ideal' the 
strictly upper triangular subalgebra.  

The following is an example of an operator $*$-algebra which is a maximal subdiagonal algebra in the sense of Arveson \cite{Arv2}
within the hyperfinite II$_1$ factor $R$.   One could call this the {\em  hyperfinite upper triangulars}.   In $M_{2^n}$ consider conjugation by the order reversing permutation 
matrix which for convenience we write as $u_n$ (in the notation above it is $u_{2^n}$).
Then we have $u_{n+1} (x \oplus x) u_{n+1} = (u_n x u_n) \oplus (u_n x u_n)$ for $x \in M_{2^n}$.
It follows that $(u_n x^* u_n)$ gives rise to  a well defined period 2 $*$-automorphism on the union ${\mathcal C}$ 
of the copies on $M_{2^n}$ .
This extends by density to a period 2 $*$-automorphism $\theta$ of the CAR algebra $B$, and this gives an involution 
on the closure $A$ of the union of the  upper triangular matrices in  $M_{2^n}$ for all
$n \in \Ndb$,  since $\theta(A) \subset A^*$.      So $A$ is an operator $*$-algebra.    A similar construction 
using the transpose in place of $*$ gives a linear involution on $A$.
Note that for the 
normalized trace $$\tau_n(\theta(x) y^*) = \tau_n(u_n x u_n y^*) = \tau_n(x \theta(y)^*), \qquad x, y \in M_{2^n},$$
so that $\theta$ extends to a symmetry $U$ on the Hilbert space of the GNS representation of the trace of $B$.
Since $U A U \subset A^*$, it is easy to argue that the weak* closure $N$ of $A$ is a dual operator $*$-algebra
inside ${\mathcal R}$, the hyperfinite II$_1$ factor.   Similarly one has a   linear involution on $N$. We claim that $N$ is a subdiagonal algebra in the sense of Arveson.
If $\Phi_n : M_n \to D_n \subset M_n$ is the canonical projection onto the matrices supported on main diagonal in $M_n$,
then $\Phi_{n+1}(x \oplus x) = \Phi_n(x) \oplus  \Phi_n(x)$ for $x \in M_{2^n}$.  Thus we obtain a trace preserving projection $\Psi$ 
from the 
CAR algebra $B$ onto the `main diagonal' part $D_0$ of $B$.   Indeed $\tau(\Psi(x) y ) = \tau(\Psi(xy)) = \tau (xy)$ for $x \in B, y \in D_0$.
   On the other hand the canonical  trace preserving  conditional expectation $\Phi$ from $R$ onto the `main diagonal' part $D$ of $R$
restricts to a trace preserving  conditional expectation from $B$ onto $D_0$, so by the unicity of the
 trace preserving  normal conditional expectation  we get that $\Phi$ extends $\Psi$.
Since $\Psi$ is multiplicative on $B$, by density $\Phi$ is a homomorphism onto $D$.  Also since 
$A + A^*$ is clearly dense in $B$, by density we have 
$N + N^*$ is weak* dense in $R$, so $N$ is a maximal subdiagonal algebra in the sense of Arveson.

\subsection{The algebra of an involutive operator space} \label{uxy}   Let $X$ be an operator system or selfadjoint subspace of
$B(H)$, and consider, as in 2.2.10 in \cite{BLM}, 
 $${\mathcal U}(X) = \biggl\{ \left[ \begin{array}{ccl} \lambda_1 & x\\ 0
&\lambda_2 \end{array} \right] \, :\, x\in X,\ \lambda_1
,\lambda_2 \in\Cdb\biggr\} \subset B(H \oplus H),$$
where $\lambda_1$ and $\lambda_2$ stand for the operators
$\lambda_1 I_H$ and $\lambda_2 I_H$ respectively.  Note that 
${\mathcal U}(X)$ may be regarded as a subspace of the Paulsen system ${\mathcal U}(X) + {\mathcal U}(X)^*$.
Give ${\mathcal U}(X)$  the involution that we gave the upper triangular matrices, namely 
$(u_2 \otimes I_H) a^* (u_2 \otimes I_H)$, where $u_2$ is the usual permutation matrix on $\Cdb^2$.
Then ${\mathcal U}(X)$ is an operator $*$-algebra.

More abstractly we define an  {\em operator} $*$-{\em space}
  to be an operator space $X$ with a 
period 2 conjugate linear  bijection $* : X \to X$ satisfying $\| [a_{ji}^* ] \|
= \| [ a_{ij}] \|.$    As shown in the introduction to \cite{BKNW} if $u : X \to B(K)$ is
a linear complete isometry, then 
$$\Theta(x) \; = \;  \left[ \begin{array}{ccl} 0  & u(x) \\ u(x^*)^*
& 0 \end{array} \right] \in B(K \oplus K), \qquad x \in X, $$  is a $*$-linear 
complete
isometry.  So $X$ `is' a selfadjoint subspace of
$B(H)$ for a Hilbert space $H$, and gives rise to an operator $*$-algebra ${\mathcal U}(X)$ as above. 

A special case that is sometimes used is when $X$ is an `involutive Banach space'.   Then $X$ with its Min or Max operator space structure (see 1.2.21
and 1.2.22 in \cite{BLM}) 
will be  an operator $*$-space.
There is a similar construction for the other types of `involution' mentioned at the start of Section \ref{inv}.    Namely,
if $X$ is an operator space with an involution satisfying the conditions at the start of Section \ref{inv} of  one of these four types, but with no multiplicativity or anti-multiplicativity condition
assumed, then the operator algebra 
${\mathcal U}(X)$ may be given an operator algebra involution of the matching type.

\subsection{The algebra of an involutive bimodule}  There is an operator module version of the last example.   Since we plan to study   operator modules in an involutive setting later we will be brief here.   In \cite{BKM} 
a kind of involutive operator module is studied that is quite different to,
 and much more interesting than, the ones below, although the representation in the next result has 
a superficial similarity inspired by the `standard forms' considered there.   

\begin{theorem}
Let $A$ be an approximately unital operator $\ast$-algebra and let $X$ be an operator $\ast$-space in the sense of Subsection {\rm \ref{uxy}} above.
Suppose that $X$ is a nondegenerate operator $A$-$A$-bimodule in the 
sense of {\rm \cite[Chapter 3]{BLM}} such that 
$(ax)^{\dagger}=x^{\dagger}a^{\dagger},$ for any $a\in A$ and $x\in X.$ Then there exist a Hilbert space $H,$ a completely isometric linear map $\sigma: X\to B(H),$ and a nondegenerate completely isometric homomorphism $\pi$ of $A$ on $H,$ and selfadjoint unitary $u$ on $H,$ such that
\begin{align*}
&\sigma(x^{\dagger})=u\sigma(x)^{\ast}u \,\,\quad {\rm and} \quad \pi(a^{\dagger})=u\pi(a)^{\ast}u,\\
&\pi(a)\sigma(x)=\sigma(ax)\quad {\rm and} \quad \sigma(x)\pi(a)=\sigma(xa),
\end{align*}
for all $a\in A$ and $x\in X.$
\end{theorem}

\begin{proof}
By \cite[Theorem 3.3.1, Lemma 3.3.5]{BLM}, there exist a Hilbert space $H_0,$ a completely isometric linear map $\phi:X\to B(H_0),$ and nondegenerate completely completely isometric homomorphism $\Theta$ of $A$ such that
\begin{align*}
\Theta(a)\phi(x)=\phi(ax)\quad \mbox{and}\quad \phi(x)\Theta(b)=\phi(xb),
\end{align*}
 for all $a, b\in A$ and $x\in X.$  We consider the Hilbert space $H :=H_0\oplus H_0$ and the completely isometric homomorphism $\pi: A\to B(H)$ given by
 \[\pi(a)=\begin{pmatrix}
 \Theta(a) & 0 \\
0&\Theta(a^{\dg})^{\ast}
\end{pmatrix}\]
and the complete isometry $\sigma: X\to B(H)$ given by
\[\sigma(x)=\begin{pmatrix}
\phi(x)& 0 \\
 0& \phi(x^{\dagger})^*
\end{pmatrix}.\]
The self-adjoint unitary $u=\begin{pmatrix}
0 & 1 \\
1& 0
\end{pmatrix}$ implements the relations
$$\sigma(x^{\dagger})=u\sigma(x)^{\ast}u \,\,\quad {\rm and} \quad \pi(a^{\dagger})=u\pi(a)^{\ast}u.$$
Moreover, it is easy to see that $\pi(a)\sigma(x)=\sigma(ax)$ and $\sigma(x)\pi(a)=\sigma(xa).$
\end{proof}

Let $X$ be a nondegenerate operator $A$-$A$-bimodule 
over an approximately unital operator $*$-algebra $A$,
of the type characterized in the last theorem.
For $H, \pi, \sigma$ as in that theorem consider
  $${\mathcal U}_A(X) = \biggl\{ \left[ \begin{array}{ccl} \pi( a_1)  & \sigma(x) \\ 0
& \pi(a_2) \end{array} \right] \, :\, x\in X,\ a_1, a_2 \in A \biggr\} \subset B(H \oplus H).$$
    This is an operator $*$-algebra, with involution
$a \mapsto (u_2 \otimes I_H) a^* (u_2 \otimes I_H)$, where $u_2$ is the usual permutation matrix on $\Cdb^2$,  similarly to Example \ref{uxy}.

There are similar constructions for some of the other type of involutions
listed at the start of Section \ref{inv}.  

\medskip

We mention a few final examples.  If $A$ is an operator $*$-algebra then so are   $A^\circ$ and $A^\star$.
For any operator algebra $A$ we have that 
$A \oplus A^\star$ is an operator $*$-algebra with
obvious involution.   Example 1.9 (2) of \cite{BKM} is a  natural example of an  operator $*$-algebra inside the reduced free group $C^*$-algebra.
This may be modified to give examples in the group von Neumann algebra or full group $C^*$-algebra.
 Finally, there are many examples of period 2 automorphisms in  the literature of operator algebras, although they are usually not very explicit.
For example they sometimes occur as a special case of
finite group actions on operator algebras,
or the $\Zdb_2$-action case of crossed product operator algebras.

\section{Operator $\ast$-algebras}

Henceforth in our paper for specificity our involutive algebras will be operator $\ast$-algebras.  As said earlier, we leave the case 
of the remaining material for the other kinds of involutions to the reader.  We remark that the $C^{\ast}$-algebras which are operator $\ast$-algebras are 
exactly the $\Zdb_2$-{\em graded $C^*$-algebras}.  

An {\em involutive ideal} or $\dagger${\em -ideal} in an operator algebra with involution $\dagger$ is an ideal $J$ with $J^\dagger \subset J$. 

\begin{proposition}\label{qdoai}
 Let $A$ be an operator $\ast$-algebra. Suppose $J$ is a closed $\dagger$-ideal, then $J$ and $A/J$ are operator $\ast$-algebras.
\end{proposition}
\begin{proof}   This follows from the matching fact for  operator algebras \cite[Proposition 2.3.4]{BLM}, and the computation 
$$\Vert [a_{ji}^\dagger + J ] \Vert
	\leq \Vert [a_{ji}^\dagger + x_{ji}^\dagger] \Vert
	\leq \Vert [a_{ij}+ x_{ij}] \Vert, \qquad x_{ij} \in J,$$ so that
	$\Vert [a_{ji}^\dagger + J ] \Vert \leq \Vert [a_{ij}+ J] \Vert$ for $a_{ij} \in A$.
 Similarly, we have $\Vert [a_{ij}+ J] \Vert\leq \Vert [a_{ji}^\dagger + J ] \Vert.$
\end{proof}

{\bf Remark.}   There are $*$-algebra variants of  the usual consequences of   the matching fact  in 
operator algebra theory.  For example one may deduce easily from Proposition \ref{qdoai} following the method in e.g.\ 1.2.30, 2.3.6, 2.3.7 in \cite{BLM},
that one may {\em interpolate} between operator $*$-algebras.   Indeed 
suppose that $(A_0, A_1)$ is a compatible couple of Banach $*$-algebras which happen to be
operator $*$-algebras.   Just like in the general operator space case \cite[1.2.30]{BLM}, let $\Sl$ be the strip of all complex numbers $z$  with $0\leq \re z\leq 1$ and let $\Fl=\Fl(A_0, A_1)$ be the space of all bounded and continuous functions $f:\Sl \to A_0+A_1$ such that the restriction of $f$ to the interior of $\Sl$ is analytic, and such that the maps $t\to f(it)$ and $t\to f(1+it)$ belong to $C_0(\Rdb; A_0)$ and $C_0(\Rdb, A_1)$ respectively. For any $f\in \Fl,$ the function $f^{\dg}$ is defined by $f^{\dagger}(z)=f(\bar{z})^* \in \Fl.$ 
The last asterisk here is the involution on $A_0+A_1$.
Then $\Fl(A_0, A_1)$ with the operator
space considered in 1.2.30  in \cite{BLM} is an operator $\ast$-algebra with the involution $\dagger$.   
For any $0\leq \theta \leq 1,$ let $\Fl_{\theta}(A_0, A_1)$ be the two-sided closed ideal of all $f\in\Fl$ for which $f(\theta)=0.$  This is $\dagger$-selfadjoint.  The interpolation space $A_{\theta}=[A_0, A_1]_{\theta}$ is the subspace of $A_0+ A_1$ formed by all $x=f(\theta)$ for some $f\in\Fl$. As operator spaces, the interpolation space $A_{\theta} \cong \Fl(A_0, A_1)/\Fl_{\theta}(A_0, A_1)$ through the map $\pi: f\mapsto f(\theta).$ It is easy to see that $\pi$ is $\dagger$-linear.  By Proposition \ref{qdoai}, the quotient $A_{\theta}\cong \Fl(A_0, A_1)/\Fl_{\theta}(A_0, A_1)$ is an operator $\ast$-algebra.

\subsection{Contractive approximate identities}

\begin{lemma}\label{dcai}
Let $A$ be an operator $\ast$-algebra. Then the following are equivalent: 
\begin{itemize}
\item[(i)] $A$ has a cai.
\item[(ii)]$A$ has a $\dagger$-selfadjoint cai.
\item[(iii)] $A$ has a left cai.
\item[(iv)] $A$ has a right cai. 
\item[(v)] $A^{\ast\ast}$ has an identity of norm $1.$ 	
\end{itemize}
\end{lemma}
\begin{proof}
(i) $\Rightarrow$ (ii) \ If $(e_t)$ is a cai for $A,$ then $(e_t^{\dagger})$ is also a cai for $A.$ Let $f_t=(e_t+e_t^{\dagger})/2,$ then $(f_t)$ is a $\dagger$-selfadjoint cai for $A.$

(iii) $\Rightarrow$ (iv) \ If $(e_t)$ is a left cai for $A,$ then $(e_t^{\dagger})$ is a right cai. Analogously, it is easy to see that (iv) $\Rightarrow$ (iii.)

(iv) $\Rightarrow$ (v)  \ By a well-known fact in operator algebra that if $A$ has a left cai and right cai, then $A^{\ast\ast}$ has an identity of norm $1$ (see e.g. \cite[Proposition 2.5.8]{BLM}).

That (ii) $\Rightarrow$ (i), and 
(i) $\Rightarrow$ (iii), are obvious.  That (v) $\Rightarrow$ (i) follows from Proposition 2.5.8 in \cite{BLM}.
\end{proof}

\begin{corollary}\label{cdgcai12}
If $A$ is an operator $\ast$-algebra with a countable cai $(f_n),$ then $A$ has a countable $\dagger$-selfadjoint cai in $\frac{1}{2}{\mathfrak F}_A.$	
\end{corollary}

\begin{proof}
By \cite[Theorem 1.1]{BRI}, $A$ has a cai $(e_t)$ in $\frac{1}{2}{\mathfrak F}_A.$ Denote $e_t'=\frac{e_t+e_t^{\dagger}}{2},$  then $(e_t')$ is also a cai in $\frac{1}{2}{\mathfrak F}_A.$ Choosing $t_n$ with $\Vert f_n e_{t_n}'-f_n\Vert +\Vert e_{t_n}'f_n-f_n\Vert <2^{-n},$ it is easy to see that $(e_{t_n}')$ is a countable $\dagger$-selfadjoint cai in $\frac{1}{2}{\mathfrak F}_A.$
\end{proof}

\begin{corollary}
If $J$ is a closed two-sided $\dagger$-ideal in an operator $\ast$-algebra $A$ and if $J$ has a cai, then $J$ has a $\dagger$-selfadjoint cai $(e_t)$ with $\Vert 1-2e_t\Vert\leq 1$ for all $t,$ which is also quasicentral in $A$.	
\end{corollary}

\begin{proof}
By the proof of Corollary \ref{cdgcai12}, we know that $J$ has a  $\dagger$-selfadjoint cai, denoted $(e_t),$ in $\frac{1}{2}{\mathfrak F}_A.$ 
The rest is as in the proof of \cite[Corollary 1.5]{BRI}.
\end{proof}

Let $A$ be an operator algebra (possibly not unital). Then the {\em left (resp.\, right) support projection} of an element $x$ in $A$ is the smallest projection $p\in A^{\ast\ast}$ such that $px=x$ (resp.\,  $xp=x$), if such a projection exists (it always exists if $A$ has a cai, see e.g. \cite{BRI}). If the left and right support projection exist, and are equal, then we call it the {\em support projection}, written $s(x).$
\begin{theorem}\cite[Corollary 3.4]{BRII}
For any operator algebra $A,$ if $x\in {\mathfrak r}_A$ and $x\neq 0,$ then the left support projection of $x$ equals the right support projection, and equals the weak* limit of $(x^{1/n})$.   It also equals $s(y),$ where $y=x(1+x)^{-1}\in \frac{1}{2}{\mathfrak F}_A$. Also, $s(x)$ is open in $A^{\ast\ast}$.
\end{theorem}

\begin{proposition} \label{sdaggi}
In an operator $*$-algebra $A$, ${\mathfrak F}_A$ and ${\mathfrak r}_A$ are $\dagger$-closed,
and if $x \in {\mathfrak r}_A$ we have $(x^\dg)^{1/n} = (x^{1/n})^{\dg}$ for  $n \in \Ndb$, and $s(x)^\dg = s(x^{\dg})$.
If $x \in {\mathfrak F}_A$ then  $s(x) \vee s(x^\dagger) = s(x + x^\dagger)$. 
In particular if $x$ is $\dagger$-selfadjoint then so is $s(x)$.  \end{proposition}
\begin{proof}   Applying $\dagger$ we see that
$\| 1 - x \| \leq 1$ implies $\| 1 - x^\dagger \| \leq 1$.   For the  $\dagger$-invariance of  ${\mathfrak r}_A$ note that this is easy to see for a $C^*$-cover $B$ with compatible involution (Definition \ref{ivltcmptb}), and then one may use the fact that 
${\mathfrak r}_A = A \cap {\mathfrak r}_B$.   Similarly if $x \in  {\mathfrak r}_A$ and $\sigma$ is the
 $*$-automorphism on $B$ inducing $\dagger$ then $\sigma(x^{1/n}) = \sigma(x)^{1/n}$ clearly.
 Applying $*$,  it follows that 
$(x^{\dg})^{1/n} =  (x^{1/n})^\dagger.$     Note that if $x \in {\mathfrak F}_A$ then this may be seen more explicitly  since  $x^{1/n}$ may be written as a power series in $1-x$ with real coefficients.  Then
$s(x)^\dagger = (w^{\ast}\lim_n x^{1/n})^\dagger =  s(x^{\dg})$.  
 The  $s(x + x^\dagger)$ assertion follows from e.g.\ the proof of \cite[Proposition 2.14]{BRI}.  
\end{proof}	

Thus the theory of real positivity studied in many of the first authors recent papers will have good involutive variants.   

 See the end of Section \ref{inv}   for notation used in the following result and proof.

\begin{corollary} \label{fbab} For any operator $\ast$-algebra $A,$ if $x\in {\mathfrak r}_A$ is $\dagger$-selfadjoint, then  $a={\mathfrak F}(x)=x(1+x)^{-1}\in \frac{1}{2}{\mathfrak F}_A$ is
$\dagger$-selfadjoint, and $\overline{xA}=\overline{aA}=s(x)A^{\ast\ast}\cap A$ is a  right ideal in $A$ with 
a $\dagger$-selfadjoint left cai. Also, $\overline{xAx}=\overline{aAa}$ is a  $\dagger$-selfadjoint HSA whose support
projection is $\dagger$-selfadjoint.  
\end{corollary}
\begin{proof} 
It is an exercise that  $a=x(1+x)^{-1}$ is $\dagger$-selfadjoint,
and is in $\frac{1}{2}{\mathfrak F}_A$ by the previous
result.
Since 
 $(a^{1/n})$ is  $\dagger$-selfadjoint by a fact in  the last proof, $(a^{1/n})$ 
 serves as a $\dg$-selfadjoint left cai for $\overline{aA}.$ Besides, $\overline{aAa}$ is $\dagger$-selfadjoint and the weak* limit of $(a^{1/n})$ is $s(a).$ The rest follows from \cite[Corollary 3.5]{BRII}.	
\end{proof}

We remark that by the operator algebra theory,   $\overline{xAx}$ is the HSA  matching the right ideal $\overline{xA}$ under the bijective
correspondence between these objects  (see e.g. \cite{BHN,BRI,BRII}).   See also Section 5.1 below for the involutive variant of this correspondence.

\begin{lemma}
If $x\in {\mathfrak F}_A,$ with $x\neq 0,$ then the operator $\ast$-algebra generated by $x$, denoted $\osa{x},$ has a cai.	Indeed, the operator $\ast$-algebra $\osa{x}$ has a $\dagger$-selfadjoint sequential cai belonging to $\frac{1}{2}{\mathfrak F}_A.$
\end{lemma}
\begin{proof}
 If $x\in {\mathfrak F}_A,$ then $x^{\dg}\in {\mathfrak F}_A$ as we proved above.    
Denote $B=C^{\ast}_e(A),$ then $p = s(x)\vee s(x^{\dg}) = s(x + x^\dagger)$ in $B^{\ast\ast}$ is in $\osa{x}^{\ast\ast}.$  
Clearly 
$p x=x p=x$ and $p x^{\dg}=x^{\dg} p=x^{\dg}.$
Therefore, $p$ is an identity in $\osa{x}^{\ast\ast}.$ By \cite[Theorem 2.5.8]{BLM}, $\osa{x}$ has a cai. 

Moreover, since $\osa{x}$ is separable,  by \cite [Corollary 2.17]{BRI}, there exists $a\in{\mathfrak F}_A$ such that $s(a)=1_{\osa{x}^{\ast\ast}}$. Therefore $\osa{x}$ has a countable $\dagger$-selfadjoint cai by applying to \cite[Theorem 2.19]{BRI} and Corollary \ref{cdgcai12}. 
\end{proof}

We write $x  \preccurlyeq y$ if $y - x \in {\mathfrak r}_A$.   The ensuing  `order theory'  in the involutive case is largely similar to 
the operator algebra case from \cite{BRord}.  For example:

\begin{theorem} \label{havin}  Let  $A$ be an operator $*$-algebra which generates a $C^*$-algebra
$B$ with compatible involution $\dagger$, and let ${\mathcal U}_A = \{ a \in A : \Vert a \Vert < 1 \}$.  The following are equivalent:
\begin{itemize} \item [(1)]   $A$ is approximately unital.
 \item [(2)]  For any $\dagger$-selfadjoint positive $b \in {\mathcal U}_B$ there exists  $\dagger$-selfadjoint $a \in {\mathfrak c}_A$
with $b \preccurlyeq a$.
 \item [(2')]  Same as {\rm (2)}, but also $a \in \frac{1}{2}  {\mathfrak F}_A$ and `nearly positive' in the sense of the introduction to  {\rm \cite{BRord}}:
we can make it as close  in norm as we like to an actual positive element.
  \item [(3)]   For any pair of $\dagger$-selfadjoint elements 
$x, y \in {\mathcal U}_A$ there exist   nearly positive $\dagger$-selfadjoint
$a \in \frac{1}{2}  {\mathfrak F}_A$
with $x \preccurlyeq a$ and $y \preccurlyeq a$.
\item [(4)]    For any $\dagger$-selfadjoint $b \in {\mathcal U}_A$  there exist   nearly positive $\dagger$-selfadjoint
$a \in \frac{1}{2}  {\mathfrak F}_A$
with $-a \preccurlyeq b \preccurlyeq a$.
\item [(5)] For any $\dagger$-selfadjoint $b \in {\mathcal U}_A$  there exist $\dagger$-selfadjoint 
$x, y \in  \frac{1}{2}  {\mathfrak F}_A$ 
with $b = x-y$.
\item [(6)]  ${\mathfrak r}_A$ is a generating cone, indeed any $\dagger$-selfadjoint element in 
$A$ is a difference of  two $\dagger$-selfadjoint elements in ${\mathfrak r}_A$.
\item [(7)]  Same as {\rm (6)} but with ${\mathfrak r}_A$ replaced by $\Rdb_+ \, {\mathfrak F}_A$.
\end{itemize}  \end{theorem}

\begin{proof}   (1) $\Rightarrow$ (2') \  By the proof in \cite[Theorem 2.1]{BRord}  for any $\dagger$-selfadjoint positive $b \in {\mathcal U}_B$ there exists  $c  \in \frac{1}{2}  {\mathfrak F}_A$ and nearly positive with $b \leq {\rm Re} \  c$.   Hence it is easy to see that 
 $b \leq {\rm Re} \,  (c^\dagger)$ and $b \leq {\rm Re} \, a$ where  $a = (c + c^\dagger)/2$.

(2')  $\Rightarrow$ (3) \ By $C^*$-algebra theory there exists  positive $b \in {\mathcal U}_B$ with
${\rm Re} \, x$ and ${\rm Re} \, y$ both $\leq b$.   It is easy to see that $b^\dagger = \sigma(b) \geq 0$.
Then ${\rm Re} \, x  \leq b^\dagger$, so that
${\rm Re} \, x  \leq (b + b^\dagger)/2$.   Similarly for $y$.
Then  apply (2') to obtain $a$ from $(b + b^\dagger)/2$.

The remaining implications follow the proof in  \cite[Theorem 2.1]{BRord}  but using tricks similar to the ones we have used so far
in this proof.   We leave the details to the reader. 
\end{proof}

{\bf Remark.}  Similarly as in Proposition 2.6 in \cite{BRord}, but using our 
Theorem \ref{havin} (3) in the proof in place of the 
matching result referenced there,  one can show that
for an approximately unital operator $\ast$-algebra $A$, the $\dagger$-selfadjoint 
elements of norm $< 1$ in $\frac{1}{2}{\mathfrak F}_A$ is a directed set in 
the $\preccurlyeq$ ordering, and is a cai for $A$ which is increasing
in this ordering.   

\medskip

The following is a version of the Aarnes-Kadison Theorem for operator $\ast$-algebras.

\begin{theorem}[Aarnes-Kadison type Theorem] If $A$ is an operator $\ast$-algebra then the following are equivalent:
\begin{itemize}
\item[(i)] There exists a $\dagger$-selfadjoint $x\in{\mathfrak r}_A$ with $A$ the norm closure of $xAx.$
\item[(ii)]	There exists a $\dagger$-selfadjoint $x\in{\mathfrak r}_A$ with $A$ the norm closure of $xA$,
which equals the norm closure of $Ax.$
\item[(iii)] There exists a $\dagger$-selfadjoint $x\in{\mathfrak r}_A$ with $s(x)=1_{A^{\ast\ast}}.$
\item[(iv)] $A$ has a countable $\dagger$-selfadjoint cai.
\item[(v)]  $A$ has a $\dg$-selfadjoint and strictly real positive element.
\end{itemize}
Indeed these are all equivalent to the same conditions with `$\dagger$-selfadjoint' removed.	
\end{theorem}
\begin{proof}   In (i)--(iii) we can assume that $x\in{\mathfrak F}_A$ by replacing it with the $\dagger$-selfadjoint
element $x(1+x)^{-1} \in  \frac{1}{2} {\mathfrak F}_A$ (see \cite[Section 2.2]{BRord}).  
Then the  equivalence of (i)--(iv) follows as in \cite[Lemma 2.10 and Theorem 2.19]{BRI}, for (iv) using that 
$x^{\frac{1}{n}}$ is  $\dagger$-selfadjoint as we said in the proof of Proposition \ref{sdaggi}. 
 Similarly (v) follows from these by  \cite[Lemma 2.10]{BRI},
and the converse follows since  strictly real positive elements have support projection $1$
(see \cite[Section 3]{BRord}).   The final assertion follows since 
if $A$ has a countable cai, then $A$ has a $\dagger$-selfadjoint countable cai (Lemma \ref{dcai}). 
\end{proof}

\subsection{Cohen factorization for operator $\ast$-algebras and their modules}
The Cohen factorization theorem is a crucial tool for Banach algebras, operator algebras and their modules. In this section we will give a variant that works for operator $\ast$-algebras and their modules.
Recall that if $X$ is a Banach space and $A$ is a Banach algebra then $X$ is called a {Banach $A$-module} if there is a module action $A\times X\to X$ which is a contractive linear map. If $A$ has a bounded approximate identity $(e_t)$ then we say that $X$ is nondegenerate if $e_t x\to x$ for $x\in X.$ A {\em Banach $A$-bimodule} is both a left and a right Banach $A$-module such that $a(xb)=(ax)b.$

The following is an operator $\ast$-algebra version of the Cohen factorization theorem:
\begin{theorem}
If $A$ is approximately unital operator $\ast$-algebra, and if $X$ is a nondegenerate Banach $A$-module(resp.\ $A$-bimodule), if $b\in X$ then there exists an element $b_0\in X$ and a $\dg$-selfadjoint $a\in {\mathfrak F}_A$ with $b=ab_0$ (resp.\ $b=ab_0a$). Moreover if $\Vert b\Vert<1$ then $b_0$ and $a$ may be chosen of norm $<1.$ \end{theorem}
\begin{proof}
In \cite[Theorem 4.1]{PED}, the $a$ is constructed from convex combinations of elements in a cai, and in our case the cai may be chosen $\dagger$-selfadjoint by Lemma \ref{dcai}. The details are  left as an exercise to the reader. 
\end{proof}

\subsection{Multiplier algebras}
\begin{theorem}\label{lefmult}
Let $A$ be an approximately unital operator $\ast$-algebra. Then the following algebras are completely isometrically isomorphic:
\begin{itemize}
\item[(i)] $LM(A)=\{\eta\in A^{\ast\ast}: \eta A\subset A\},$ 

\item[(ii)] $LM(\pi)=\{T \in B(H): T\pi(A)\subset\pi(A)\},$ 
where $\pi$ is a  nondegenerate completely isometric representation  of $A$ on a Hilbert space $H$ such that there exists an order 2 $\ast$-automorphism $\sigma: B(H)\to B(H)$ satisfying $\sigma(\pi(a))^{\ast}=\pi(a^{\dagger})$ for any $a\in A,$ 
\item[(iii)] the set of completely bounded right $A$-module maps $CB_A(A).$ \end{itemize}	
\end{theorem}
\begin{proof}
See \cite[Theorem 2.6.3]{BLM}.	
\end{proof}
\begin{definition}
Let $A$ be an approximately unital operator $\ast$-algebra. Then we define
\begin{itemize}
\item[(i)] $RM(A)=\{\xi\in A^{\ast\ast}: A\xi\subset A\};$
\item[(ii)] $RM(\pi)=\{S\in B(H): \pi(A)S\subset\pi(A)\},$ for any nondegenerate completely isometric representation $\pi$ of $A$ on a Hilbert space $H$ and there exists order-2 $\ast$-automorphism $\sigma: B(H)\to B(H)$ satisfies $\sigma(\pi(a))^{\ast}=\pi(a^{\dagger})$ for any $a\in A;$
\item[(iii)] the set of completely bounded left $A$-module maps, which we denote as $_A{CB}(A).$
\end{itemize}	
\end{definition}

\begin{corollary}
Let $A$ be an approximately unital operator $\ast$-algebra. Then
\begin{itemize}
\item[(a)] $\eta\in LM(A)$ if and only if $\eta^{\dagger}\in RM(A),$ where $\eta, \eta^{\dagger}\in A^{\ast\ast}$ and $\dagger$ is the involution in $A^{**};$
\item[(b)] $T\in LM(\pi)$ if  and only if $T^{\dagger}\in RM(\pi),$ where $T^{\dagger}=\sigma(T)^{\ast};$	
\item[(c)] $L\in CB_A(A)$ if  and only if $L^{\dagger}\in$ $_A{CB}(A),$ where the map $L^{\dagger}$ is defined by $L^{\dagger}(a)=L(a^{\dagger})^{\dagger}.$
\end{itemize}
\end{corollary}
\begin{proof}
We just give the proof of (b). Suppose that $T\in LM(\pi),$ then
\begin{align*}
\pi(a)T^{\dagger}=\sigma(\pi(a^{\dagger}))^{\ast}\sigma(T)^{\ast}=(\sigma(T\pi(a^{\dagger})))^{\ast}\in \sigma(\pi(A))^{\ast}\subset \pi(A).	
\end{align*}
Thus, $T^{\dagger}\in RM(\pi).$ Similarly, if $T^{\dagger}\in RM(\pi)$ then $T\in LM(\pi).$
\end{proof}

We consider pairs $(D,\mu)$ consisting of a unital operator $\ast$-algebra $D$ and a completely isometric $\dagger$-homomorphism $\mu:A\to D,$ such that $D\mu(A)\subset\mu(A),\ \mu(A)D\subset\mu(A)$.  We use the phrase {\em multiplier operator $\ast$-algebra} of $A,$ and write $M(A),$ for any pair $(D,\mu)$ which is completely  $\dagger$-isometrically $A$-isomorphic to $\Ml(A)=\{x\in A^{\ast\ast}: xA\subset A\ \mbox{and}\ Ax\subset A\}.$ Note that by Lemma \ref{bidu}
the inclusion of $A$ in $A^{**}$ is a $\dagger$-homomorphism, hence the canonical map $i:A\to \Ml(A)$, is a $\dagger$-homomorphism. From this it follows that there is a unique involution on $M(A)$ for which $i$ is a $\dagger$-homomorphism.  
\begin{proposition}
Suppose that $A$ is an approximately unital operator $\ast$-algebra. If 
$(D,\mu)$ is a left multiplier operator algebra of $A,$ then the closed 
subalgebra
$$\{d\in D:\mu(A)d\subset \mu(A)\}$$
of $D,$ together with the map $\mu$, is a multiplier operator $\ast$-algebra of 
$A.$
\end{proposition}
\begin{proof}
Let $E$ denote the set $\{d\in D:\mu(A)d\subset \mu(A)\}$. By \cite[Proposition 
2.6.8]{BLM}, we know that $E$ is a multiplier operator algebra of $A.
$ Thus, there exists a completely isometric surjective homomorphism 
$\theta:\Ml(A)\to E$ such that $\theta\circ i_A=\mu.$ Now we may define an 
involution on $E$ by $d^{\dagger}=\theta(\eta^{\dagger})$ if 
$d=\theta(\eta).$ Then it is easy to check that $E$ is an operator $\ast$-algebra which is completely $\dagger$-isometrically 
$A$-isomorphic to $\Ml(A).$
\end{proof}
\begin{example}
Let $A= A_1(\Ddb)$, the functions in the disk algebra vanishing at $1$, which is the
norm closure of $(z-1) A(\Ddb)$, and let
$B = \{ f \in C(\Tdb) : f (1) = 0 \}$. By the nonunital variant of the Stone-Weierstrass theorem,
$B$ is generated as a $C^*$-algebra by $A$. Indeed $B = C^*_e(A)$, since any closed ideal of $B$ is the set of functions that vanish on a closed set in the circle containing $1$. Also for any $z_0 \in \Tdb$, $z_0 \neq 1$, there is a function in $A$ that peaks at $z_0$, if necessary by the noncommutative Urysohn lemma for approximately unital operator algebras \cite{BNII}. So the involution on $A$ descends from the natural involution on $B$.
It is easy to see, for example by examining the bidual of $B^{**}$ and noticing that $A$ and $B$ have a
common cai, that $M(A) = \{ T \in M(B) : TA \subset A \} =
\{ g\in C_b(\Tdb \setminus \{1 \}) : g (z-1) \in A(\Ddb) \}$.
For such $g$, since the negative Fourier coefficients of
$k=g(1-z)$ are zero, the negative Fourier coefficients of $g$ are constant, hence zero by the Riemann-Lebesgue lemma.  
Thus $g$ is in $H^\infty$, and has an analytic extension to the open disk.   Viewing $g$ as a  function $h$ on $\bar{\Ddb} \setminus \{1 \}$
we have $h = k/(z-1)$ for some $k \in A(\Ddb)$.  
So $M(A)$ consists of the bounded continuous functions on $\bar{\Ddb} \setminus \{1 \}$ that are analytic in the open disk, with involution  $\overline{f(\bar{z})}.$
\end{example}

Let $A, B$ be approximately unital operator $\ast$-algebras. A completely contractive $\dagger$-homomorphism $\pi: A\to M(B)$ will be called a {\em multiplier-nondegenerate $\dagger$-morphism}, if $B$ is a nondegenerate bimodule with respect to the natural module action of $A$ on $B$ via $\pi.$ This is equivalent to saying that for any cai $(e_t)$ of $A,$ we have $\pi(e_t)b\to b$ and $b\pi(e_t)\to b$ for $b\in B.$
\begin{proposition}\label{multext}
If $A, B$ are approximately operator $\ast$-algebras, and if $\pi:A\to M(B)$ is a multiplier-nondegenerate $\dagger$-morphism then $\pi$ extends uniquely to a unital completely contractive $\dagger$-homomorphism $\hat{\pi}: M(A)\to M(B).$ Moreover $\hat{\pi}$ is completely isometric if  and only if $\pi$ is completely isometric.	
\end{proposition}

\begin{proof}
Regard $M(A)$ and $M(B)$ as $\dagger$-subalgebras of $A^{\ast\ast}$ and $B^{\ast\ast}$ respectively. Let $\tilde{\pi}:A^{\ast\ast}\to B^{\ast\ast}$ be the unique $w^{\ast}$-continuous $\dagger$-homomorphism extending $\pi.$ From \cite[Proposition 2.6.12]{BLM}, we know that $\hat{\pi}=\tilde{\pi}(\cdot)_{\vert_{M(A)}}$ is the unique bounded homomorphism on $M(A)$ extending $\pi,$ and $\hat{\pi}(M(A))\subset M(B).$

Let $(e_t)$ be a $\dg$-selfadjoint cai for $A.$ Then for any $\eta\in M(A),$ $\eta e_t \in A$ and $\eta e_t\xrightarrow{w^{\ast}} \eta$ .
Hence
$$\hat{\pi}(\eta)=w^{\ast}-\lim_{t}\pi(\eta e_t).$$
On the other hand, $(\eta e_t)^{\dg}\to \eta^{\dg},$ which implies
$$\hat{\pi}(\eta^{\dg})=w^{\ast}-\lim_{t}\pi((\eta e_t)^{\dg})=w^{\ast}-\lim_{t}\pi((\eta e_t))^{\dg}.$$
Since the involution on $A^{\ast\ast}$ is $w^{\ast}$-continuous, we get that
$$\hat{\pi}(\eta^{\dg})=\hat{\pi}(\eta)^{\dg}.$$
The rest follows from \cite[Proposition 2.6.12]{BLM}.
\end{proof}

\subsection{Dual operator $\ast$-algebras}

\begin{definition}
	Let $M$ be a dual operator algebra (that is, an operator algebra with an operator space predual--see
\cite[Section 2.7]{BLM}) and operator $\ast$-algebra such that the involution on $M$ is weak* continuous. Then $M$ is called a {\em dual operator $\ast$-algebra}.	
\end{definition}

We will identify any two dual operator $\ast$-algebras $M$ and $N$ which are $w^*$-homeomorphically and completely $\dagger$-isometrically isometric. 
\begin{proposition}
	Let $M$ be a dual (possibly nonunital) operator $\ast$-algebra.
	\begin{itemize}
		\item[(1)] The $w^*$-closure of a $\ast$-subalgebra of $M$ is a dual operator $\ast$-algebra.	
		\item[(2)] The unitization of $M$ is also a dual operator $\ast$-algebra.
	\end{itemize}
\end{proposition}
\begin{proof}
	For (1), the weak* -closure of $\ast$-subalgebra of $M$ is a dual operator algebra by \cite[Proposition 2.7.4 (4)]{BLM}. 
	
	For (2), suppose that $M$ is a nonunital operator $\ast$-algebra and write $I$ for the identity in $M^1.$ Suppose that $(x_t)_t$  and $(\lambda_t)_t$ are nets in $M$ and $\Cdb$ respectively, with $(x_t+\lambda_t I)$ converging in $w^*$-topology. By applying a nonzero weak* continuous
functional annihilating $M$, it is easy to see that $(\lambda_t)_t$ converges in $\Cdb.$ It follows that $(x_t)_t$ converges in $M$ in the $w^{*}$-topology. Thus, $(x_t+\lambda_t I)^{\dagger}$ converges in $M^1,$ in the $w^{*}$-topology. The rest follows immediately from \cite[Proposition 2.7.4 (5)]{BLM}.\end{proof}

\begin{proposition}
	Let $A$ be an operator $\ast$-algebra, and $I$ any cardinal. Then $\Kdb_I(A)^{\ast\ast}\cong \Mdb_I(A^{\ast\ast})$ as dual operator $\ast$-algebras.	
\end{proposition}
\begin{proof}
	The canonical embedding $A\subset A^{\ast\ast}$ induces a completely isometric $\dagger$-homomorphism $\theta: \Kdb_I(A)\to \Kdb_I(A^{\ast\ast})\subset \Mdb_{I}(A^{\ast\ast}).$  Notice that the involutions on $\Kdb_I(A)^{\ast\ast},$ $\Mdb_I(A^{\ast\ast})$ are $w^*$-continuous and $\Kdb_I(A)^{\ast\ast}\cong \Mdb_I(A^{\ast\ast})$ as operator algebras. Thus, $\Kdb_I(A)^{\ast\ast}\cong \Mdb_I(A^{\ast\ast})$ as dual operator $\ast$-algebras.	
\end{proof}

\begin{lemma}  \label{uxd} 
	If $X$ is a weak* closed  selfadjoint subspace of $B(H)$ for a Hilbert space $H$, 
then $\Ul(X)$ as defined as in Example \ref{uxy} is  a dual operator $\ast$-algebra.
\end{lemma}
\begin{proof}
	We leave this as an exercise to the reader.
\end{proof}

In connection with the last result we note that any operator $*$-space $X$ in the sense of \ref{uxy}, which is a dual 
operator space, and whose involution is weak* continuous, may be embedded weak* homeomorphically, via a $*$-linear 
complete
isometry, as a  weak* closed  selfadjoint subspace of $B(H)$ for a Hilbert space $H$.   So  $\Ul(X)$ is  a dual operator $\ast$-algebra,
again by Lemma \ref{uxd}.    To see this simply use the proof in \ref{uxy}, taking $u$ there to be a 
weak* homeomorphic complete isometry from $X$ into $B(H)$.

 The last result can be used to produce counterexamples concerning dual operator $*$-algebras, such as algebras with two distinct preduals, etc.
Similarly one may use the $\Ul(X)$ construction to easily obtain an example of a dual operator algebra which is an operator $\ast$-algebra, but the involution is not weak*-continuous.   We omit the details.

Recall that in \cite{BS} the maximal
$W^*$-algebra $W^*_{\rm max}(M)$ was defined for
  unital dual operator algebras $M$.
If $M$ is a dual operator algebra but
is not unital we define $W^*_{\rm max}(M)$ to be the
von Neumann subalgebra of $W^*_{\rm max}(M^1)$ generated
by the copy of $M$. Note that it has the desired
universal property: if $\pi : M \to N$ is a weak* continuous
completely contractive homomorphism into a von Neumann algebra
$N$, then by the normal version of Meyer's theorem (see \cite[Proposition
2.7.4 (6)]{BLM}) we may
extend to a weak* continuous
completely contractive unital homomorphism
$\pi^1 : M^1 \to N$.  Hence by the universal property
of $W^*_{\rm max}(M^1)$, we may extend further
to a normal unital $*$-homomorphism from $W^*_{\rm max}(M^1)$
into $N$.  Restricting to $W^*_{\rm max}(M)$
we have shown that   there exists a normal $*$-homomorphism
$\tilde{\pi} : W^*_{\rm max}(M) \to N$ extending $\pi$.

\begin{proposition}
Let $B = W^*_{\rm max}(M)$. Then $M$ is a dual operator $*$-algebra if and only if
there exists an order two $*$-automorphism
$\sigma : B \to B$ such that $\sigma(M) = M^*$.
In this case the involution on $M$ is  $a^{\dagger} = \sigma(a)^*$.	
\end{proposition}

 \begin{proof}
This follows from a simple variant of the part of the
proof of Theorem \ref{opinvch}	that we did prove above,
where one ensures that all maps there are weak* continuous.
\end{proof}

\begin{proposition}
For any dual operator $*$-algebra $M,$ there is a Hilbert space $H$ (which may be taken to be $K \oplus K$ if $M \subset B(K)$ as a dual 
operator algebra completely 
isometrically), and a symmetry (that is, a selfadjoint unitary) $u$ on $H$, and a weak* continuous completely isometric homomorphism $\pi : M \to B(H)$
such that $\pi(a)^* = u \pi(a^\dagger) u$ for $a \in M$.	
\end{proposition}
\begin{proof}
This is a tiny modification of the proof in \cite[Proposition 1.12]{BKM}, 
beginning with a weak* continuous completely isometric homomorphism 
$\rho : M \to B(K)$ 
and checking that the $\pi : M \to B(H)$
produced in that proof is weak* continuous.
 \end{proof}

 Similarly one has variants of the last two propositions valid for a dual operator algebra with a weak* continuous 
linear involution.   These will look like Theorem \ref{lininvch} but with the representations and isomorphisms also being weak* homeomorphisms.  
 Thanks to this, one will also have dual operator algebra  variants of Corollaries \ref{chsym}, \ref{chsym2}, and \ref{chsym3}, basically characterizing
symmetric  dual operator algebras.  
In these we are assuming also that $A$ is a dual operator algebra, 
and in the conclusions the representations and isomorphisms will also be weak* homeomorphisms.     For example:

\begin{corollary} \label{chsym4}  The  weak* closed algebra generated by any operator on a Hilbert space is
 isometrically isomorphic via a weak* homeomorphism to the  weak* closed   algebra generated by a complex symmetric operator 
on another Hilbert space.
\end{corollary}

\begin{proposition}
	Let $M$ be a dual operator $\ast$-algebra, and let $I$ be a $w^*$-closed $\dagger$-ideal. Then $M/I$ is a dual operator $\ast$-algebra.
\end{proposition}
\begin{proof}
From \cite[Proposition 2.7.11]{BLM}, we know that $M/I$ is a dual operator algebra. As dual operator spaces, $M/I\cong (I_{\perp})^{\ast},$ from which it is easy to see that the involution on $M/I$ is $w^*$-continuous.	
\end{proof}

\begin{lemma}\label{delta}   If $A$ is an operator $*$-algebra then $\Delta(A) = A \cap A^*$ (adjoint 
taken in any containing $C^*$-algebra; see {\rm 2.1.2} in {\rm  \cite{BLM}})  is a 
$C^{\ast}$-algebra and $\Delta(A)^\dagger = \Delta(A)$. 
\end{lemma}   \begin{proof}  That 
  $\Delta(A)$ does not depend on the particular containing $C^*$-algebra may be found in e.g.\ 2.1.2 in \cite{BLM}. 
as is the fact that it is spanned by its selfadjoint (with respect to the usual involution) elements.  If $A$ is also an operator $*$-algebra then $\Delta(A)$ is invariant under $\dagger$.  Indeed suppose that
 $B$ is a  $C^{\ast}$-cover of $A$ with compatible involution
coming from a $*$-automorphism $\sigma$ as usual. If $x = x^* \in \Delta(A)$ 
then $\sigma(x)$ is also selfadjoint, so is in $\Delta(A)$.  This
holds by linearity for any $x \in \Delta(A)$.  So 
$\Delta(A)^\dagger = \Delta(A)$.   \end{proof}

If $M$ is a dual operator algebra then $\Delta(M) = M \cap M^*$,  is a $W^*$-algebra (see e.g.\ 2.1.2 in \cite{BLM}).
If $M$ is a dual operator $*$-algebra then $\Delta(M)$ is a dual operator   $\ast$-algebra, indeed
it is a $W^*$-algebra with an extra involution $\dagger$ inherited from $M$.  

\begin{proposition}\label{jmdg}   
Suppose that $M$ is a dual operator $*$-algebra. 
Suppose that  $(p_i)_{i\in I}$ is a collection of projections in $M$. Then $(\wedge_{i\in I}\,  p_i)^{\dagger}=\wedge_{i\in I}\, p_i^{\dagger}$ and $(\vee_{i\in I}\, p_i)^{\dagger}=\vee_{i\in I}\, p_i^{\dagger}.$  
\end{proposition}
\begin{proof}  By the analysis above the proposition we may assume that
$M$ is a $W^*$-algebra with an extra involution $\dagger$, which
is weak* continuous and is of the form $x^\dagger = \sigma(x)^*$ for
a weak* continuous period 2 $*$-automorphism $\sigma$ of $M$.   
If $p_i$ and $p_j$ are two projections in $M$, then $p_i\wedge p_j=\lim_n (p_ip_j)^n=\lim_n (p_jp_i)^n.$ By the weak$^{\ast}$-continuity of involution on $M$ we have
$$(p_i\wedge p_j)^{\dagger}=\lim_n [(p_jp_i)^n]^{\dagger}=\lim_n (p_i^{\dagger}p_j^{\dagger})^n=p_i^{\dagger}\wedge p_j^{\dagger}.$$
Thus for any finite subset $F$ of $I,$ we have $(\wedge_{i\in F}\, p_i)^{\dagger}=\wedge_{i\in F}\, p_i^{\dagger}.$
Note that the net $(\wedge_{i\in F}\, p_i)_F$ indexed by the directed set of finite subsets $F$ of $I,$ is a decreasing net with limit $\wedge_{i\in I}\, p_i.$ 
We have 
$$(\wedge_{i\in I}\, p_i)^{\dagger}=\lim_F \, (\wedge_{i\in F}\, p_i)^{\dagger}=\lim_F\, (\wedge_{i\in F}\, p_i^{\dagger})=\wedge_{i\in I} \, p_i^{\dagger}.$$
by weak* continuity of involution. The statement about 
suprema of projections follows by taking orthocomplements.
 \end{proof}

\subsection{Involutive M-ideals}
An $M$-projection $P$ on a Banach space with involution $X$ is called a {\em $\dagger$-$M$-projection} if $P$ is $\dagger$-preserving. A subspace $Y$ of $X$ is called a {\em $\dagger$-$M$-summand} if $Y$ is the range of a $\dagger$-$M$-projection. Such range is $\dagger$-closed. Indeed, if $y\in Y,$ then  $y=P(x)$ for some $x\in X.$ Thus, $y^{\dagger}=P(x)^{\dagger}=P(x^{\dagger})\in Y.$  A subspace $Y$ of $E$ is called an {\em involutive M-ideal} or a {\em $\dagger$-$M$-ideal} in $E$ if $Y^{\perp\perp}$ is a $\dagger$-$M$-summand in $E^{\ast\ast}.$
If $X$ is an operator $\ast$-space, then an $M$-projection is called a {\em complete $\dagger$-$M$-projection} if the amplification $P_n$ is a $\dagger$-$M$-projection on $M_n(X)$ for every $n\in \Ndb.$ Similarly, we could define {\em complete $\dagger$-$M$-summand}, {\em complete $\dagger$-$M$-ideal}, {\em left $\dagger$-$M$-projection},  {\em right $\dagger$-$M$-summand} and {\em right $\dagger$-$M$-ideal}. 
\begin{proposition}\label{dmsmi}
Let $X$ be an operator $\ast$-space. 
\begin{itemize}
\item[(1)] A linear idempotent $\dagger$-linear map $P: X\to X.$ 
 $P$ is a left $\dagger$-$M$-projection if and only if it is a right $\dagger$-$M$-projection, and these imply $P$ is a complete $\dagger$-$M$-projection.	
\item[(2)] A subspace $Y$ of $X$ is a complete $\dagger$-$M$-summand if and only if it is a left $\dagger$-$M$-summand if  and only if it is a right $\dagger$-$M$-summand.
\item[(3)] A subspace $Y$ of $X$ is a complete $\dagger$-$M$-ideal if and only if it is a left $M$-$\dagger$-ideal if and only if it is a right $\dagger$-$M$-ideal.

\end{itemize}

\end{proposition}
\begin{proof}
(1) \ If $P$ is a left $\dagger$-$M$-projection, then the map
\[\sigma_p(x)=\begin{pmatrix}
P(x)  \\
x-P(x)
\end{pmatrix}\]	
is a completely isometry from $X$ to $C_2(X).$ Also,
\begin{align*}
\Vert x^{\dagger}\Vert&=\Vert \sigma_P(x^{\dagger})\Vert=\biggl\Vert\begin{pmatrix}
P(x^{\dagger})  \\
x^{\dagger}-P(x^{\dagger})
\end{pmatrix}\biggr\Vert= \biggl\Vert \begin{pmatrix}
P(x^{\dagger})&0 \\
x^{\dagger}-P(x^{\dagger})&0
\end{pmatrix}\biggr\Vert\\
&=\biggl\Vert \begin{pmatrix}
P(x)^{\dagger})&0 \\
x^{\dagger}-P(x)^{\dagger}&0
\end{pmatrix} \biggr\Vert=\biggl\Vert \begin{pmatrix}
P(x)&x-P(x) \\
0&0
\end{pmatrix}^{\dagger} \biggr\Vert\\
&=\biggl\Vert \begin{pmatrix}
P(x)&x-P(x) \\
0&0
\end{pmatrix}\biggr\Vert
=\biggl\Vert \begin{pmatrix}
P(x), x-P(x)\end{pmatrix}\biggr\Vert=\Vert x\Vert.
\end{align*}
One can easily generalize this to matrices, so that $P$ is a right $\dagger$-$M$-projection. Similarly, if $P$ is is a right $\dagger$-$M$-projection then $P$ is a left $\dagger$-$M$-projection. By Proposition 4.8.4 (1) in \cite{BLM}, we know that $P$ is a complete $\dagger$-$M$-projection.

(2) \  It follows from (1) and \cite[Proposition 4.8.4 (2)]{BLM}. Now (3) is also clear.
\end{proof}
\begin{theorem}
Let $A$ be an approximately unital operator $\ast$-algebra. 
\begin{itemize}
\item[(i)] The right $\dagger$-$M$-ideals are the $\dagger$-$M$-ideals in $A$, which are also the complete $\dagger$-$M$-ideals. These are exactly the approximately unital closed  $\dagger$-ideals in $A.$
\item[(ii)] The right $\dagger$-$M$-summands are the $\dagger$-$M$-summands in $A,$ which are also the complete $\dagger$-$M$-summands. These are exactly the principal ideals $Ae$ for a $\dagger$-selfadjoint central projection $e\in M(A).$  \end{itemize}
\end{theorem}
\begin{proof}
(ii) \ By Proposition \ref{dmsmi} (2), the right $\dagger$-$M$-summands are exactly the complete $\dagger$-$M$-summands. Moreover, by \cite[Theorem 4.8.5 (3)]{BLM}, the $M$-summands in $A$ are exactly the complete $M$-summands. If $D$ is a $\dagger$-$M$-summand, then $D$ is a complete $\dagger$-$M$-summand and there exists a central projection $e\in M(A)$ such that $D=eA.$ Then $D^{\perp\perp}=eA^{\ast\ast}$ and $e$ is an identity for $D^{\perp\perp}$. Also, $e^{\dagger}$ serves as an identity in $D^{\perp\perp}$, so that $e=e^{\dagger}.$

(i) \ By a routine argument, the results follow as in \cite[Theorem 4.8.5 (1)]{BLM} and Proposition \ref{dmsmi} (3).
\end{proof}

\section{Involutive hereditary subalgebras, ideals, and $\dagger$-projections}

\subsection{Involutive hereditary subalgebras}
Throughout this section $A$ is an operator $\ast$-algebra (possibly not approximately unital). Then $A^{\ast\ast}$ is an operator $\ast$-algebra.
We  recall from the end of Section \ref{inv}  the definition of  an {\em open} (or $A$-open) and {\em closed}  projection in $A^{**},$ {\em hereditary subalgebras} of $A$ (or HSA's),  and  {\em support projections} of  HSA's (which are the same as $A$-open projections).
   If $p$ is $A$-open, then $p^{\dagger}$ is also $A$-open.    Indeed, if $x_t \in A$ with $x_t=px_t=x_tp=px_tp\to p\, \, {\rm weak^*},$ then $x_t^{\dagger}=p^{\dagger}x_t^{\dagger}=x_t^{\dagger}p^{\dagger}=p^{\dagger}x_t^{\dagger}p^{\dagger}\to p^{\dagger}\, \, {\rm weak^*},$ which means $p^{\dagger}$ is also open.

If $p$ is $\dagger$-selfadjoint and open, then we say $p$ is {\em $\dagger$-open} in $A^{**}$.   This happens if  and only if there exists a $\dagger$-selfadjoint net $(x_t)$ in $A$ with
 $$x_t=px_t=x_t p=p x_t p\to p \,\, {\rm weak^*}.$$  To see this replace $(x_t)$ in the last paragraphs by $((x_t+x_t^\dagger)/2)$.
 If also  $A$ is approximately unital  then we say that $p^\perp = 1-p$ is $\dagger$-closed.
If $p$ is $\dagger$-open in $A^{**}$ then  clearly
$$D=pA^{\ast\ast}p \cap A=\{a\in A: a=ap=pa=pap\}$$ 
is a closed $\dagger$-subalgebra of $A.$ We call such a $\dagger$-subalgebra $D$ is an {\em involutive hereditary subalgebra} or a {\em $\dagger$-hereditary subalgebra} of $A$ (or,\ $\dagger$-HSA).   HSA's are inner ideals: that is $DAD \subset D$.

In the following statements, we often omit the proof details where are similar to usual operator algebras case (see e.g. \cite{BLM, BHN,BRI,BRII})
\begin{proposition}\label{dhsa1}
A subalgebra $D$ of an operator $*$-algebra $A$ is a HSA and $D^{\dagger}\subset D$ if and only if $D$ is a $\dg$-HSA.	
\end{proposition}
\begin{proof}
One direction is trivial.  Conversely, if $D$ is a HSA, then $D=pA^{\ast\ast}p\cap A,$ for some open projection $p\in A^{\ast\ast}.$ Here, $p\in D^{\perp\perp}$ and $p$ is an identity for $D^{\perp\perp}.$ If also $D$ is $\dg$-selfadjoint, then $p^{\dg}\in D^{\perp\perp}$ also serves as identity. By uniqueness of identity for $D^{\perp\perp},$ then $p=p^{\dagger}.$  
\end{proof}

\begin{proposition}\label{dhsainner}
A subspace of an operator $\ast$-algebra $A$	 is a $\dg$-HSA if and only if it is an approximately unital $\dg$-selfadjoint inner ideal.
\end{proposition}
\begin{proof}
If $J$ is a $\dg$-HSA, then $J$ is an approximately unital $\dg$-selfadjoint  inner ideal.

If $J$ is an approximately unital $\dg$-selfadjoint  inner ideal, then by Proposition \ref{dhsa1} $J$ is a HSA and $\dg$-selfadjoint which means that J is a $\dg$-HSA.
\end{proof}

{\bf Remark.}
If $J$ is an approximately unital ideal or  inner ideal of operator $\ast$-algebra, we cannot necessarily expect $J$ to be $\dg$-selfadjoint. For example, let $A(\Ddb)$ be the disk algebra and $$A_i(\Ddb)=\{f: f\in A(\Ddb), f(i)=0\}.$$ 
Then $A_i(\Ddb)$ is an approximately unital ideal but obviously it is not $\dg$-selfadjoint.

\medskip

The following is another characterization of $\dagger$-HSA's.
\begin{corollary}\label{dghsa}
Let $A$ be an operator $\ast$-algebra and suppose that $(e_t)$ is a $\dg$-selfadjoint net in $\ball(A)$ such that $e_te_s\to e_s$ with $t.$ Then
$$\{x\in A: xe_t\to x, e_t x\to x\}$$
is a $\dg$-HSA of $A.$ Conversely, every $\dg$-HSA of $A$ arises in this way.
\end{corollary}
\begin{proof}
Let $J=\{x\in A: xe_t\to x, e_t x\to x\}.$ Then it is easy to see that $J$ is an inner ideal and $J^{\dagger}\subset J.$ By Proposition \ref{dhsainner}, $J$ is a $\dagger$-HSA. Conversely, if $D$ is a $\dagger$-HSA and $(e_t)$ is a $\dg$-selfadjoint cai for $D,$ then 
$$D=pA^{**}p\cap A=\{x\in A: xe_t\to x, e_t x\to x\},$$
 where $p$ is the weak* limit of $(e_t).$
 \end{proof}

Closed right ideals $J$ of an operator $\ast$-algebra $A$ possessing a $\dagger$-selfadjoint left cai will be called {\em r-$\dagger$-ideals}. Similarly, closed left ideals $J$ of an operator $\ast$-algebra $A$ possessing a $\dagger$-selfadjoint right cai will be called l-$\dagger$-{\em ideals}.
Note that there is a bijective correspondence between r-$\dagger$-ideals and l-$\dagger$-ideals, namely $J\to J^{\dagger}.$ 
For $C^*$-algebras r-$\dagger$-ideals are precisely the closed right ideals, and there is an obvious bijective correspondence between r-$\dagger$-ideals and l-$\dagger$-ideals, namely $J\to J^*$.
  
\begin{theorem}\label{dgrlhsa}
Suppose that $A$ is an operator $\ast$-algebra (possibly not approximately unital), and $p$ is a $\dagger$-projection in $A^{\ast\ast}.$ Then the following are equivalent:
\begin{itemize}
\item[(i)] $p$ is $\dg$-open in $A^{\ast\ast}.$
\item[(ii)]	$p$ is the left support projection of an r-$\dagger$-ideal of $A.$
\item[(iii)] $p$ is the right support projection of an l-$\dagger$-ideal of $A.$
\item[(iv)] $p$ is the support projection of a $\dg$-hereditary algebra of $A.$
\end{itemize}
\end{theorem}
\begin{proof}
The equivalence of (i) and (iv) is just the definition of being $\dg$-open in $A^{\ast\ast}.$ 

Suppose (i).  If $p$ is $\dg$-open then $p$ is the support projection for some $\dg$-HSA $D.$ Let $(e_t)$ be a $\dg$-selfadjoint cai for $D,$ then $p=w^{\ast}$-$\lim_{t}e_t.$ 	Let 
$$J=\{x\in A: e_tx\to x\},$$
then $J$ is a right ideal of $A$ with $\dg$-selfadjoint left cai $(e_t)$ and $p$ is the left support projection of $J.$

Suppose (ii).  If $p$ is the left support projection of an r-$\dagger$-ideal $J$ of $A$ with $\dg$-selfadjoint left cai $(e_t),$ then $J=pA^{\ast\ast}\cap A.$ Therefore $J^{\dagger}=A^{\ast\ast}p\cap A,$ which is an l-$\dagger$-ideal and $p$ is the right support projection of $J^{\dagger}.$

Suppose (iii).  If $p$ is the right support projection of an l-$\dagger$-ideal of $A$ with $\dg$-selfadjoint right cai $(e_t),$ then $p=$weak*-$\lim_t e_t=p^{\dg},$ which means that $p$ is $\dg$-open. 

Similarly we can get the equivalence between (i) and (iii).
\end{proof}

If $J$ is an operator $\ast$-algebra with an $\dg$-selfadjoint left cai $(e_t),$ then we set 
$$\Ll(J)=\{a\in J: ae_t\to a\}.$$
\begin{corollary}
A subalgebra of an operator $\ast$-algebra $A$ is $\dg$-hereditary if and only if it equals $\Ll(J)$ for an r-$\dagger$-ideal $J$. Moreover the correspondence $J\mapsto \Ll(J)$ is a bijection from the set of r-$\dagger$-ideals of $A$ onto the set of $\dg$-HSA's of $A.$ The inverse of this bijection is the map $D\to DA.$ Similar results hold for the l-$\dagger$-ideals of $A.$	
\end{corollary}
\begin{proof}
If $D$ is a $\dg$-HSA, then by Corollary \ref{dghsa}, we have 
$$D=\{x\in A: xe_t\to x, e_tx\to x\},$$
where $(e_t)$ is a  $\dg$-selfadjoint cai for $D.$ Set $J=\{x\in A: e_tx\to x\},$ then $J$ is an r-$\dagger$-ideal with $D=\Ll(J).$ 

Conversely, if $J$ is an r-$\dagger$-ideal and $(e_t)$ is a $\dagger$-selfadjoint left cai for $J,$ then 
$$D=\{x\in A: xe_t \to x, e_tx \to x\}$$ 
is a $\dg$-HSA  by Corollary \ref{dghsa}, and $D=\Ll(J).$ The remainder is as in \cite[Corollary 2.7]{BHN}. 
\end{proof}

As in the  operator algebra case  \cite[Corollary 2.8]{BHN}, if $D$ is a $\dg$-hereditary subalgebra of an operator $\ast$-algebra $A,$ and if $J=DA$ and $K=AD,$ then $JK=J\cap K=D.$	  Also as in the  operator algebra case  \cite[Theorem 2.10]{BHN}, any $\dg$-linear functional  on a HSA $D$
of an approximately unital operator $\ast$-algebra $A$  has a unique $\dg$-linear Hahn-Banach extension to $A$.   This is because 
if $\varphi$ is any Hahn-Banach extension to $A$, then $\overline{\varphi(x^\dagger)}$ is another, so these must be equal
by \cite[Theorem 2.10]{BHN}.

\begin{proposition}
Let $D$ be an approximately unital $\dg$-subalgebra 	of an approximately unital operator $\ast$-algebra $A.$ The following are equivalent:
\begin{itemize}
\item[(i)] $D$ is a $\dg$-hereditary subalgebra of $A.$	
\item[(ii)] Every completely contractive unital $\dg$-linear map from $D^{\ast\ast}$ into a unital operator $\ast$-algebra $B$, has a unique completely contractive unital $\dg$-extension from $A^{\ast\ast}$ into $B.$
\item[(iii)] Every completely contractive  $\dg$-linear map $T$ from $D$ into a unital weak* closed operator $\ast$-algebra $B$ such that $T(e_t)\to 1_B$ weak* for some cai $(e_t)$ for $D$ has a unique completely contractive weakly $\dg$-extension $\tilde{T}$ from $A$ into $B$ with $\tilde{T}(f_s)\to 1_B$ weak* for some(or all) cai $(f_s)$ for $A.$
\end{itemize}
\end{proposition}
\begin{proof} 
Let $e$ be the identity of $D^{\ast\ast}.$ Obviously, $e$ is $\dagger$-selfadjoint. 
 If (iii) holds, then the inclusion from $D$ to $D^{\perp\perp}$ extends to a unital complete $\dg$-contraction $T: A\to D^{\ast\ast}\subset eA^{\ast\ast}e.$ The map $x\to exe$ on $A^{\ast\ast}$ is also a completely contractive unital $\dg$-extension of the inclusion map $D^{\ast\ast}\to eD^{\ast\ast}e.$ It follows from the hypothesis that these maps coincide, and so $eA^{\ast\ast}e=D^{\ast\ast},$ which implies that $D$ is a $\dg$-HSA. 
The rest is left as an exercise to the reader, being very similar to the proof of \cite[Proposition 2.11]{BHN}.
\end{proof}

\subsection{Support projections and $\dagger$-HSA's}

\begin{lemma}\label{sdhsa}
If $(J_i)$ is a family of r-$\dagger$-ideals in an operator $\ast$-algebra $A,$ with matching family of $\dagger$-HSA's $(D_i),$ and if $J$ is the norm closure of $\sum_i \, J_i$ then the $\dg$-HSA matching $J$ is the $\dg$-HSA $D$ generated by the $(D_i)$.
\end{lemma}
\begin{proof}
This follows from the matching operator algebra result, since e.g.\ any $\dagger$-HSA is a HSA.
\end{proof}

In the next result, e.g.\ $\overline{zA}$ denotes the {\em norm closure} of $zA$.

\begin{proposition}\label{csrdihsa}
Let $A$ be an operator $\ast$-algebra (not necessarily with an identity or approximate identity). Suppose that $(x_k)$ is a sequence of $\dg$-selfadjoint elements in $\mathfrak{F}_A,$ and $\alpha_k\in (0,1]$ add to $1.$ Then the closure of the sum of the r-$\dagger$-ideals $\overline{x_k A},$ is the r-$\dagger$-ideal $\overline{zA},$ where $z=\sum_{k=1}^{\infty} \alpha_kx_k\in \mathfrak{F}_A.$ Similarly, the $\dagger$-HSA generated by all the $\overline{x_kAx_k}$ equals $\overline{zAz}.$	 \end{proposition}
\begin{proof}
By the matching operator algebra result,
the right ideal $\overline{zA}$ is the closure of the sum of the right ideals $\overline{x_kA}.$  If $z\in \mathfrak{F}_A$ is $\dagger$-selfadjoint then $\overline{zA}$ is an r-$\dagger$-ideal.
\end{proof}

If $S\subset A,$ define $S_{\dagger}$ to be the set of $\dagger$-selfadjoint elements in $S.$ 
\begin{lemma}\label{spjoin}
Let $A$ be an operator $\ast$-algebra, a subalgebra of a $C^{\ast}$-algebra $B.$
\begin{itemize}
\item[(i)] The support projection of a $\dagger$-HSA $D$ in $A$ equals $\vee_{a\in ({\mathfrak F}_D)_{\dagger}}s(a)$ (which equals $\vee_{a\in ({\mathfrak r}_D)_{\dagger}}s(a)$).
\item[(ii)]	 The support projection of an r-$\dagger$-ideal $J$ in $A$ equals $\vee_{a\in ({\mathfrak F}_J)_{\dagger}}s(a)$ (which equals $\vee_{a\in ({\mathfrak r}_J)_{\dagger}}s(a)$).
\end{itemize}
\end{lemma}
\begin{proof}
(i) \ Suppose $p$ is the support projection of $D,$ then $p=\vee_{b\in {\mathfrak F}_D}s(b)=\vee_{b\in {\mathfrak r}_D}s(b)$ by the operator algebra 
case we are generalizing. Thus, 
$$p\geq \vee_{a\in ({\mathfrak r}_D)_{\dagger}} \, s(a)\geq \vee_{a\in ({\mathfrak F}_D)_{\dagger}} \, s(a).$$
For any $b\in {\mathfrak F}_D,$ we have $b^{\dagger}\in {\mathfrak F}_D$ and 
$s(b)\vee s(b^{\dagger})=s(\frac{b+b^{\dagger}}{2})$, by Proposition \ref{sdaggi}. 
Hence, 
$$p=\vee_{b\in {\mathfrak F}_D}s(b)\leq \vee_{a\in ({\mathfrak F}_D)_{\dagger}}s(a).$$
Therefore, $p\leq \vee_{a\in ({\mathfrak F}_D)_{\dagger}}s(a)\leq \vee_{a\in ({\mathfrak r}_D)_{\dagger}}s(a).$ 

(ii) \ This is similar.
\end{proof}

\begin{lemma}\label{mindgshsa}
For any operator $\ast$-algebra $A,$ if $E\subset ({\mathfrak r}_A)_{\dagger},$ then the smallest $\dagger$-hereditary subalgebra of $A$ containing $E$ is $pA^{\ast\ast}p\cap A,$ where $p=\vee_{x\in E}\, s(x).$	
\end{lemma}
\begin{proof}  By  the operator algebra case we are generalizing, $pA^{\ast\ast}p\cap A$ is the smallest $\dagger$-hereditary subalgebra of $A$       containing $E$.
Conversely, if $D$ is a $\dagger$-HSA of $A$ containing $E$ then $D^{\perp\perp}$ contains $p$ by a routine argument, so $pA^{\perp\perp}p\subset D^{\perp\perp}$ and $pA^{\perp\perp}p\cap A\subset D^{\perp\perp}\cap A=D.$
 \end{proof}
\begin{corollary}
For any operator $\ast$-algebra $A,$ suppose that a convex set $E\subset {\mathfrak r}_A$ and $E^{\dagger}\subset E.$ Then the smallest hereditary subalgebra of $A$ containing $E$ is $pA^{\ast\ast}p\cap A,$ where $p=\vee_{x\in E_{\dagger}} s(x).$ Indeed, this is the smallest $\dagger$-HSA of $A$ containing $E.$
\end{corollary}
\begin{proof}  By the lemma the smallest HSA (and smallest $\dagger$-HSA) containing $E$ is $pA^{\ast\ast}p\cap A,$ where $p=\vee_{a\in E}\, s(a).$ For any $a\in E$, $\frac{a+a^{\dagger}}{2}\in E$ by convexity of $E$. Notice that $s(\frac{a+a^{\dagger}}{2})\leq p$ and $s(\frac{a+a^{\dagger}}{2})\geq s(a).$   So  $p=\vee_{x\in E_{\dagger}}s(x)$. 
\end{proof}

Again in the next several results in this section, `overline' denotes the {\em norm closure}.

\begin{theorem}
If $A$ is an operator $\ast$-algebra then $\dagger$-HSA's (resp.\ r-$\dagger$-ideals) in $A$ are precisely the sets of form $\overline{EAE}$ (resp.\ $\overline{EA}$) for some $E\subset({\mathfrak r}_A)_{\dagger}.$	The latter set is the smallest $\dagger$-HSA (resp.\ r-$\dagger$-ideal) of $A$ containing $E.$ 
\end{theorem}
\begin{proof}
If $D$ is a $\dagger$-HSA (resp.\ r-$\dagger$-ideal) and taking $E$ to be a $\dagger$-selfadjoint cai for the $\dagger$-HSA $D$ (resp.\ a $\dagger$-selfadjoint left cai for the r-$\dagger$-ideal), then the results follows immediately. 
Conversely for any $x\in ({\mathfrak r}_A)_{\dagger},$ we have $x(1+x)^{-1}\in (\frac{1}{2}{\mathfrak F}_A)_{\dagger}$ as we said in
Corollary \ref{fbab}.  Then as in \cite[Theorem 3.18]{BWj} we may assume that $E\subset(\frac{1}{2}{\mathfrak F}_A)_{\dagger}.$ Note that $D=\overline{EAE}$ is the smallest HSA  containing $E$ by \cite[Theorem 3.18]{BWj} and $D$ is $\dagger$-selfadjoint, so that $D$ is the smallest $\dg$-HSA containing $E.$ Similarly, $\overline{EA}$ is the smallest right ideal with a $\dagger$-selfadjoint left contractive identity of $A$ containing $E$. Moreover, for any finite subset $F\subset E$ if $a_F$ is the average of the elements in $F,$ then $(a_F^{1/n})$ will serve as a $\dagger$-selfadjoint left cai for $\overline{EA}.$ \end{proof}

In particular, the largest $\dagger$-HSA in  an operator $\ast$-algebra $A$ is the largest HSA in $A$, and the largest approximately unital subalgebra in $A$ (see
\cite[Section 4]{BRII}), namely $A_H = \overline{{\mathfrak r}_A A {\mathfrak r}_A}
= \overline{({\mathfrak r}_A)_{\dagger} A ({\mathfrak r}_A)_{\dagger}}$.
The latter equality follows because $A_H$ has a cai in ${\mathfrak r}_A$, hence has a  cai in $({\mathfrak r}_A)_{\dagger}$.

\begin{theorem}\label{chdghsar}
Let $A$ be an operator $\ast$-algebra (not necessarily with an identity or approximate identity.) The $\dagger$-HSA's (resp.\ r-$\dagger$-ideals) in $A$ are precisely the closures of unions of an increasing net of $\dg$-HSA's (resp.\ r-$\dagger$-ideals) of the form $\overline{xAx}$ (resp.\ $\overline{xA}$) for $x\in ({\mathfrak r}_A)_{\dagger}.$	
\end{theorem}
\begin{proof}
Suppose that $D$ is a $\dg$-HSA (resp.\ an r-$\dagger$-ideal). The set of $\dagger$-HSA's (resp.\ r-$\dagger$-ideals) $\overline{a_FAa_F}$ (resp.\ $\overline{a_F A}$) as in the last proof, indexed by finite subsets $F$ of $({\mathfrak F}_D)_{\dagger},$ is an increasing net. Lemma \ref{spjoin} can be used to show, as in \cite{BWj}, that the closure of the union of these $\dagger$-HSA's (resp.\ r-$\dagger$-ideals) is $D.$	
\end{proof}

As in the theory we are following, it follows that $\dagger$-open projections are just the sup's of a collection (an increasing net 
if desired) of $\dagger$-selfadjoint support projections $s(x)$ for 
 $\dagger$-selfadjoint  $x \in {\mathfrak r}_A$.

\begin{theorem}\label{strdghsa}
Let $A$ be any operator $\ast$-algebra (not necessarily with an identity or approximate identity). Every separable $\dagger$-HSA or $\dagger$-HSA with a countable cai (resp.\ separable r-$\dagger$-ideal or r-$\dagger$-ideal with a countable cai) is equal to $\overline{xAx}$ (resp.\, $\overline{xA}$) for some $x\in ({\mathfrak F}_A)_{\dagger}.$	
\end{theorem}
\begin{proof}
If $D$ is a $\dagger$-HSA with a countable cai, then $D$ has a countable $\dagger$-selfadjoint cai $(e_n)$ in $\frac{1}{2}{\mathfrak F}_D.$ Also, $D$ is generated by the $\dagger$-HSA's $\overline{e_n A e_n}$ so $D=\overline{xAx},$ where $x=\sum_{n=1}^{\infty} \frac{e_n}{2^n}.$ For the separable case, note that any separable approximately unital operator $\ast$-algebra has a countable cai. 
For r-$\dagger$-ideals, the result follows from the same argument. 
\end{proof}

\begin{corollary}
If $A$ is a separable operator $\ast$-algebra, then the $\dagger$-open projections in $A^{\ast\ast}$ are precisely the $s(x)$ for $x\in({\mathfrak r}_A)_{\dagger}.$	
\end{corollary}
\begin{proof}
If $A$ is separable, then so is any $\dagger$-HSA. So the result follows from Theorem \ref{strdghsa}.	
\end{proof}
\begin{corollary}
If $A$ is a separable operator $\ast$-algebra with cai, then there exists an $x\in({\mathfrak F}_A)_{\dagger}$ with $A=\overline{xA}=\overline{Ax}=\overline{xAx}.$	
\end{corollary} 

\subsection{Involutive compact projections}
Throughout this section, $A$ is an operator $\ast$-algebra. We will say that a projection $q\in A^{\ast\ast}$ is {\em compact relative to} $A$ if it is closed and $q=qx$ for some $x\in \ball(A).$ Furthermore, if $q$ is $\dagger$-selfadjoint, we say that such $q$ is an {\em involutive compact projection}, or is {\em $\dagger$-compact} in $A^{\ast\ast}.$

\begin{proposition}
 A $\dagger$-projection $q$ is compact if only if there exists a $\dagger$-selfadjoint element $a\in {\ball(A)}$ such that $q=qa.$	
\end{proposition}
\begin{proof}
One direction is trivial. Conversely if $q$ is compact, then there exists $a\in\ball(A)$ such that $q=qa.$ It is easy to argue from elementary operator theory that we have $aq=q.$ Thus, $q=q(\frac{a+a^{\dagger}}{2}).$
\end{proof} 

\begin{theorem}\label{eqrdcp}
Let $A$ be an approximately unital operator $\ast$-algebra. If $q$ is a projection in $A^{\ast\ast}$ then the following are equivalent:
\begin{itemize}
\item[(i)] $q$ is a $\dagger$-closed projection in $(A^1)^{\ast\ast},$
\item[(ii)] $q$ is $\dagger$-compact in $A^{\ast\ast}$
\item[(iii)] $q$ is closed in $A^{**}$ and there exists a $\dagger$-selfadjoint element $x\in\frac{1}{2}{\mathfrak F}_A$ such that $q=qx.$
\end{itemize}
\end{theorem}
\begin{proof}
This follows from a variant of the proof of \cite[Theorem 2.2]{BNII}: one just needs to go carefully through the proof noting that
all elements may be chosen to be  $\dagger$-selfadjoint.  
\end{proof}

\begin{corollary}
Let $A$ be an approximately unital operator $\ast$-algebra. Then the infimum of any family of $\dagger$-compact projections in $A^{\ast\ast}$ is a $\dagger$-compact projection in $A^{\ast\ast}.$ Also, the supremum of two commuting $\dagger$-compact projections in $A^{\ast\ast}$ is a $\dagger$-compact projection in $A^{\ast\ast}.$	
\end{corollary}
\begin{proof}
Note that the infimum and supremum of $\dagger$-projections are still $\dagger$-projections. Then the results follow immediately from \cite[Corollary 2.3]{BNII}.
\end{proof}

\begin{corollary}
Let $A$ be an approximately unital operator $\ast$-algebra, with an approximately unital closed $\dagger$-subalgebra $D.$ A projection $q\in D^{\perp\perp}$ is $\dagger$-compact in $D^{\ast\ast}$ if and only if $q$ is $\dagger$-compact in $A^{\ast\ast}.$	
\end{corollary}

\begin{corollary}
Let $A$ be an approximately unital operator $\ast$-algebra. If a $\dagger$-projection $q$ in $A^{\ast\ast}$ is dominated by an open projection $p$ in $A^{\ast\ast},$ then $q$ is $\dagger$-compact in $pA^{\ast\ast}p.$	
\end{corollary}

In much of what follows we use the peak projections $u(a)$ defined and studied in e.g.\ \cite{BNII,BRII}.   These may 
be defined to be projections  $q$ in $A^{**}$ which are the weak* limits of $a^n$ for some $a \in {\rm Ball}(A)$, in the case
such weak* limit exists.  We will not take the time to review the properties of $u(a)$ here.   We will however several times below 
use silently the following fact:

\begin{lemma} \label{ulemm}  If $a  \in {\rm Ball}(A)$ for an operator $\ast$-algebra $A$, and if $u(a)$ is a peak projection, with $a^n \to u(a)$ weak*, then 
$u((a + a^\dagger)/2) = u(a) \wedge u(a)^\dagger$ in $A^{**}$ and this
is a peak projection.   Indeed
$((a + a^\dagger)/2)^n \to u((a + a^\dagger)/2)$ weak*.
\end{lemma}

\begin{proof}
Clearly $(a^\dagger)^n \to u(a)^\dagger$ weak*, so that 
$u(a^\dagger) = u(a)^\dagger$ is a peak projection.   Then $u((a + a^\dagger)/2) = u(a) \wedge u(a)^\dagger$ by 
\cite[Proposition 1.1]{BNII}, and since this is a projection it is 
by \cite[Section 3]{BNII} a peak projection, is $\dagger$-selfadjoint, 
and $((a + a^\dagger)/2)^n \to u((a + a^\dagger)/2)$ weak*.
\end{proof}

The following is the
involutive variant of the version of the Urysohn lemma for approximately unital operator $\ast$-algebras in \cite[Theorem 2.6]{BNII}.  

\begin{theorem} \label{ulosacai}    Let $A$ be an approximately unital operator $\ast$-algebra. If a $\dagger$-compact projection $q$ in $A^{**}$
is dominated by a $\dagger$-open projection $p$ in $A^{\ast\ast}$, then there exists $b\in(\frac{1}{2}{\mathfrak F}_A)_{\dagger}$ with $q=qb, b=pb.$ Moreover, $q\leq u(b)\leq s(b)\leq p,$ 
and $b$ may also be chosen to be
`nearly positive' in the sense of the introduction to  {\rm \cite{BRord}}:
we can make it as close  in norm as we like to an actual positive element. 
 \end{theorem}
\begin{proof}
If $q\leq p$ as stated, then by the last corollary we know $q$ is $\dagger$-compact in $D^{\ast\ast}=pA^{\ast\ast}p,$ where $D$ is a $\dagger$-HSA supported by $p.$	By Theorem \ref{eqrdcp}, there exists a $\dagger$-selfadjoint $b\in \frac{1}{2}{\mathfrak F}_D$ such that $q=qb$ and $b=bp.$ The rest follows as in
 \cite[Theorem 2.6]{BNII}.
\end{proof}

\begin{theorem}
Suppose that $A$ is an operator $\ast$-algebra (not necessarily approximately unital), and that $q\in A^{\ast\ast}$ is a projection. The following are equivalent:
 \begin{itemize}
 \item[(1)] $q$ is $\dagger$-compact with respect to $A.$
 \item[(2)] $q$ is $\dagger$-closed with respect to $A^1$ and there exists $a\in \ball(A)_{\dagger}$ with $aq=qa=q.$
 \item[(3)] $q$ is a decreasing weak$^{\ast}$ limit of $u(a)$ for $\dagger$-selfadjoint element $a\in \ball(A).$
 \end{itemize}
\end{theorem}
\begin{proof}
(2) $\Rightarrow$ (3) \ Given (2) we certainly have $q$ compact with respect to $A$ by \cite[Theorem 6.2]{BRII}.
By \cite[Theorem 3.4]{BNII}, $q=\lim_t u(z_t),$ where $z_t\in \ball(A)$ and $u(z_t)$ is decreasing. We have  $q=q^\dagger = \lim_ t u(z_t^{\dagger}).$ Moreover, $u(z_t)\wedge u(z_t^{\dagger})=u(\frac{z_t+z_t^{\dagger}}{2}).$ Hence, $q$ is a decreasing weak* limit of $u(\frac{z_t+z_t^{\dagger}}{2})$ since the involution preserves order. 

The rest follows from \cite[Theorem 6.2]{BRII}.	
\end{proof}
\begin{corollary}
Let $A$ be a (not necessarily approximately unital) operator $\ast$-algebra. If $q$ is $\dagger$-compact then $q$ is a weak* limit of a net of $\dagger$-selfadjoint elements $(a_t)$ in $\ball(A)$ with $a_tq=q$ for all $t.$	
\end{corollary}

\subsection{Involutive peak projections}

Let $A$ be an operator $\ast$-algebra. A $\dagger$-projection $q\in A^{\ast\ast}$ is called an {\em involutive peak projection} or a {\em $\dagger$-peak projection} if it is a peak projection.

\begin{proposition}\label{scpd}
Suppose $A$ is a separable operator $\ast$-algebra (not necessarily approximately unital), then the $\dagger$-compact projections in $A^{\ast\ast}$ are precisely the peak projections $u(a),$ for some $\dagger$-selfadjoint $a\in \ball(A).$
\end{proposition}
\begin{proof}
If $A$ is separable then a projection in $A^{\ast\ast}$ is compact if and only if $q=u(a),$ for some $a\in \ball(A)$ (see \cite[Proposition 6.4]{BRII}). If $q$ is $\dagger$-selfadjoint, then   $$q=u(a^{\dagger}) =
u(a)\wedge u(a^{\dagger}) =  u((a^{\dagger}+a)/2),$$    using 
e.g.\ Lemma \ref{ulemm}.
\end{proof}

\begin{proposition}\label{dpeak}
If $a\in\frac{1}{2}{\mathfrak F}_A$ with $a^{\dagger}=a,$ then $u(a)$ is a $\dagger$-peak projection and it is a peak for $a.$	
\end{proposition}
\begin{proof}
Since  $u(a)=\lim a^n$ weak* in this case, we see that $u(a)$ is $\dagger$-selfadjoint. From \cite[Lemma 3.1, Corollary 3.3]{BNII}, we know that $u(a)$ is a peak projection and is a peak for $a.$
\end{proof}
\begin{theorem}
If $A$ is an approximately unital operator $\ast$-algebra, then
\begin{itemize}
\item[(i)] A projection $q\in A^{\ast\ast}$ is $\dagger$-compact if only if it is a decreasing limit of $\dagger$-peak projections. 
\item[(ii)] If $A$ is a separable approximately unital operator $\ast$-algebra, then the $\dagger$-compact projections in $A^{\ast\ast}$ are precisely the 
$\dagger$-peak projections.
\item[(iii)] A projection in $A^{\ast\ast}$ is a $\dagger$-peak projection in $A^{\ast\ast}$ if  and only if it is of form $u(a)$ for some $a\in (\frac{1}{2}{\mathfrak F}_A)_{\dagger}.$
\end{itemize}
\end{theorem}
\begin{proof}
(ii) \ Follows from Proposition \ref{scpd} and Proposition \ref{dpeak}.

(i) \ One direction is obvious. Conversely, let $q\in A^{\ast\ast}$ be a $\dagger$-compact projection with $q=qx$ for some $\dagger$-selfadjoint element $x\in \ball(A).$
Then $q\leq u(x).$ Now $1-q$ is an increasing limit of $s(x_t)$ for $\dagger$-selfadjoint elements $x_t\in \frac{1}{2}{\mathfrak F}_{A^1}$,
by Theorem \ref{chdghsar} and the remark after it.
 Therefore, $q$ is a decreasing weak* limit of the $q_t=s(x_t)^{\perp}=u(1-x_t).$ Let $z_t=\frac{1-x_t+x}{2},$ then $u(z_t)$ is a projection. Since $q\leq q_t$ and $q\leq u(x),$ then $q\leq u(z_t)$. Note that $z_t$ is $\dagger$-selfadjoint, so $u(z_t)=u(z_t)^{\dagger}.$ Let $a_t=z_tx\in \ball(A),$ then $u(a_t)=u(z_t)$ by the argument in \cite[Lemma 3.1]{BNII}.
 Thus, $u(a_t)=u(z_t)\searrow q$ as in that proof.  Moreover, $u(a_t^{\dagger})  = u(a_t)^{\dagger} \searrow q,$ which implies by an argument 
above  that $u(\frac{a_t+a_t^{\dagger}}{2})\searrow q.$   

(iii) \ One direction is trivial. For the other, if $q$ is a $\dagger$-peak projection, then by the operator algebra case
there exists $a\in \frac{1}{2}{\mathfrak F}_A$ such that $q=u(a).$ Let $b=(a+a^{\dagger})/2,$ then   $q=u(b)$
  by 
e.g.\ Lemma \ref{ulemm}.   
\end{proof}
\begin{corollary}
Let $A$ be an operator $\ast$-algebra. The supremum of two commuting $\dagger$-peak projections in $A^{\ast\ast}$ is a $\dagger$-peak projection in $A^{\ast\ast}.$	
\end{corollary}
\begin{lemma}
For any operator $\ast$-algebra $A,$ the $\dagger$-peak projections for $A$ are exactly the weak* limits of $a^n$ for $\dagger$-selfadjoint element $a\in \ball(A)$ if such limit exists.
\end{lemma}
\begin{proof}
If $q$ is a $\dagger$-peak projection, then there exists $a\in \ball(A)$ such that $q=u(a)$ which is also the weak* limit of $a^n.$ Since $q$ is $\dagger$-selfadjoint, by Lemma \ref{ulemm} we have $q=u(a^{\dagger}) =u(\frac{a+a^{\dagger}}{2})$, which is the weak* limit of $((a+a^{\dagger})/2)^n.$ The converse  follows from \cite[Lemma 1.3]{BRII}.
\end{proof}

{\bf Remark.}   In the theory of peak projections for operator algebras $A$ which are not necessarily approximately unital in 
 \cite[Section 6]{BRII}, there are two notions of peak projection, one stronger than the other.  If $A$ is an operator  $\ast$-algebra
one would say that  a projection is  a {\em $\dagger$-${\mathfrak F}$-peak projection} for $A$ if it is $\dagger$-selfadjoint and ${\mathfrak F}$-peak, where the latter means that it equals $u(a)$ for an $a \in \frac{1}{2}{\mathfrak F}_A$. A projection in $A^{\ast\ast}$ is  {\em $\dagger$-${\mathfrak F}$-compact} if it is a decreasing limit of $\dagger$-${\mathfrak F}$-peak projections.    We recall that $A_H$ was discussed above Theorem \ref{chdghsar}.
 One may then prove similarly to the development in  \cite[Proposition 6.5]{BRII} (with appropriate tweaks in the proofs): 
\begin{itemize}
\item[(i)] A projection $q$ in $A^{\ast\ast}$ is $\dagger$-${\mathfrak F}$-compact if only if 	$q$ is a $\dagger$-compact projection for $A_H.$
\item[(ii)] A projection in $A^{\ast\ast}$ is a $\dagger$-${\mathfrak F}$-peak projection if and only if it is a $\dagger$-peak projection for $A_H.$
\item[(iii)] If $A$ is separable then every $\dagger$-${\mathfrak F}$-compact projection in $A^{\ast\ast}$ is a $\dagger$-${\mathfrak F}$-peak projection.  
\end{itemize}

\subsection{Some interpolation results}

Item (ii) in the following should be compared with Theorem \ref{ulosacai} 
which gets a slightly better result in the case that $A$ is approximately unital.  

\begin{theorem}[Noncommutative Urysohn lemma for operator $\ast$-algebras] Let $A$ be a (not necessarily approximately unital) operator $\dagger$-subalgebra of $C^{\ast}$-algebra $B$ with a second involution $\dagger$.  Let $q$ be a $\dagger$-compact projection in $A^{\ast\ast}.$ 
\begin{itemize}
\item[(i)]	For any $\dagger$-open projection $p\in B^{\ast\ast}$ with $p\geq q,$ and any $\varepsilon>0,$ there exists an $a\in \ball(A)_{\dagger}$ with $aq=q$ and $\Vert a(1-p)\Vert <\varepsilon.$ 
\item[(ii)] For any $\dagger$-open projection $p\in A^{\ast\ast}$ with $p\geq q,$ there exists a $\dagger$-selfadjoint element $a\in \ball(A)$ with $q=qa$ and $a=pa.$
\end{itemize}
\end{theorem}
\begin{proof}
(i) \ By \cite[Theorem 6.6]{BRII} there exists $b\in \ball(A)$ such that 
$$bq=q,\,\Vert b(1-p)\Vert <\varepsilon~\mbox{and}~\Vert (1-p)b\Vert <\varepsilon.$$
Then $a=\frac{b+b^{\dagger}}{2}\in \ball(A)_{\dagger}$ does the trick, since
$$\Vert (\frac{b+b^{\dagger}}{2})(1-p) \Vert\leq \frac{1}{2}\Vert b(1-p)\Vert+\frac{1}{2}\Vert ((1-p)b)^{\dagger}\Vert<\varepsilon.$$ 

(ii) \ Apply Theorem \ref{ulosacai} in $A^1$ to obtain  a $\dagger$-selfadjoint element $a\in \ball(A^1)$, $p\in A^{\perp\perp}$ and $ap=q.$ Then $a\in A^{\perp\perp}\cap A^1=A.$  
\end{proof}

The following is an involutive variant of the noncommutative peak interpolation type result in \cite[Theorem 5.1]{BRII}.

\begin{theorem} \label{peakthang}   Suppose that $A$ is an operator $*$-algebra 
and that  $q$ is a $\dagger$-closed projection in $(A^1)^{**}$.
If $b = b^\dagger \in A$ with $b q= qb$, then $b$ achieves its distance to the
right ideal $J = \{ a \in A : q a = 0 \}$ (this is a r-$\dagger$-ideal if  $1-q \in A^{**}$),
and also achieves its distance to $\{ x \in A : x q = qx = 0 \}$ (this is a $\dagger$-HSA if  $1-q \in A^{**}$).  
If further $\Vert b q \Vert \leq 1$,
then  there exists a $\dagger$-selfadjoint
 element $g \in {\rm Ball}(A)$ with $g q =  q g = b q$.
\end{theorem}
\begin{proof}  Proceed as in the proof of \cite[Theorem 5.1]{BRII}.  The algebra 
$\tilde{D}$ is a $\dagger$-HSA in $A^1$.
Thus if $C$ is as in that proof, $C$ is $\dagger$-selfadjoint and $\tilde{D}$  is a $\dagger$-ideal in $C$.
Also $I = C \cap A$ and $D = I \cap \tilde{D}$ are  $\dagger$-selfadjoint in $C$.   Note that
 if  $x \in A$ with $x q = qx = 0$ then $x \in \tilde{D} \cap A \subset C \cap A = I$, so 
$x \in \tilde{D} \cap A \subset \tilde{D} \cap  I = D$.   So $D = \{ x \in A : x q = qx = 0 \}$.
By the  proof we are following,  there 
exists $y \in D \subset J$ such that 
$$\| b - y \| = \| b - y^\dagger \| = d(b,D) = \|  b q \| = d(b,J) \geq \| b - z \|,$$
where $z = (y + y^\dagger)/2$.   Set $g = b - z$, then    $g$ is $\dagger$-selfadjoint with $g q =  q g = b q$
(since $D$ is $\dagger$-selfadjoint), 
and $\| g \| = \|  b q \|$.  
  \end{proof}

 Theorem 4.10 in \cite{BRord} is the (noninvolutive) operator algebra
version of  the last result (and \cite[Theorem 5.1]{BRII}), but with
the additional feature that  the `interpolating element'
$g$ in the last result is also in $\frac{1}{2} {\mathfrak F}_A$.  Whence after replacing $g$ by $g^{\frac{1}{n}}$, it is
`nearly positive' in the sense of the introduction to  \cite{BRord}: 
we can make it as close  in norm as we like to an actual positive element.   As we have pointed out elsewhere, there seems
to be a mistake in  Theorem 4.10  (and hence also in 4.12)
in \cite{BRord}.  It is claimed there (and used at the end of the proof) that $D$ is approximately unital.
However this error disappears in what is perhaps the most important case, namely that  $q$ is the `perp' of 
a (open) projection in $A^{**}$.  Then $D$ is certainly a HSA in $A$, and is  approximately unital.    Hence we have  (in the
 involutive case by averaging the element produced by the
 original proof, with its `dagger'):

\begin{theorem} \label{peakth2}
 Suppose that $A$ is an operator $*$-algebra, 
 $p$ is a $\dagger$-open projection in $A^{**}$,
and $b = b^\dagger \in A$ with $bp= p b$
 and $\Vert b (1-p) \Vert \leq 1$ (where $1$ is the identity of the unitization of $A$ if $A$ is nonunital).
Suppose also that $\| (1-2b) (1-p) \| \leq 1$.  
Then  there exists a  $\dagger$-selfadjoint element $g \in \frac{1}{2} {\mathfrak F}_A
\subset {\rm Ball}(A)$ with $g (1-p) =  (1-p) g = b(1-p)$.   Such $g$ may be chosen
`nearly positive' in the sense of the introduction to  {\rm \cite{BRord}}, indeed 
it may be chosen to be as close as we like  to an actual positive element.  
\end{theorem}

\begin{theorem} \label{peakthang2}  {\rm (A noncommutative Tietze theorem)} \
 Suppose that $A$ is an operator $*$-algebra
(not necessarily approximately unital),
and that  $p$ is a $\dagger$-open projection in $A^{**}$.  Set $q = 1-p \in (A^1)^{**}$.
Suppose that $b = b^\dagger  \in A$ with $b p= pb$ and $\Vert b q \Vert \leq 1$, and that  the numerical range of 
$bq$ (in $q (A^1)^{**}q$ or $(A^1)^{**}$) 
is contained in a compact convex set $E$ in the plane satisfying $E = \bar{E}$,
where the latter is the reflection of $E$ with respect to the 
real axis.  We also suppose, by
fattening it slightly if necessary, that $E \not\subset \Rdb$.  Then there exists a $\dagger$-selfadjoint element $g \in {\rm Ball}(A)$ with $g q =  q g = b q$, 
such that  the numerical range of 
$g$ with respect to  $A^1$ is contained in $E$.
\end{theorem}

 \begin{proof}   By \cite[Theorem  4.12]{BRord},  there exists
 $a \in  {\rm Ball}(A)$ with $aq =  q a = b q$, 
such that  the numerical range of 
$a$ with respect to  $A^1$ is contained in $E$.
Then $g = (a+ a^\dagger)/2$ is  $\dagger$-selfadjoint.  Let $B$ be a unital $C^*$-cover of $A^1$ with
compatible involution $\sigma(b)^*$ as usual.  If $\varphi$ is a state of $B$ then 
$\varphi \circ \sigma$ is a state too, and so $$\varphi(a^\dagger) = \overline{\varphi(\sigma(a))} \in \bar{E} = E.$$
From this it is clear that $\varphi(g) \in E$.
 \end{proof}

\begin{corollary} Suppose that $A$ is an operator $*$-algebra
(not necessarily approximately unital),
and that  $J$ is an approximately unital closed $\dagger$-ideal in $A$. 
Suppose that $b = b^\dagger$ is an element in ${\mathfrak F}_{A/J}$ 
(resp.\ in  ${\mathfrak r}_{A/J}$).  Then there exists a $\dagger$-selfadjoint element $a$ 
 in ${\mathfrak F}_{A}$ 
(resp.\ in  ${\mathfrak r}_{A}$) with $a + J = b$.\end{corollary}

 \begin{proof} Indeed the operator $*$-algebra variant 
 of \cite[Proposition 6.1]{BRI} and \cite[Corollary 6.1]{BRI} hold.   The ${\mathfrak r}_{A/J}$
 lifting follows from the last theorem with $E= [0,K] \times [-K,K]$ and $K = \| b \|$ say.
 However both results also follow by the usual  $(a+ a^\dagger)/2$ trick.  \end{proof}

The following is the `nearly positive' case of  a simple noncommutative peak interpolation result which has implications for the unitization of an operator $\ast$-algebra.

\begin{proposition} \label{interpnp}  Suppose that $A$ is an approximately unital operator $*$-algebra, and $B$ is a $C^*$-algebra generated by $A$
with compatible involution $\dagger$.  If
$c = c^\dagger \in B_+$ with $\Vert c \Vert < 1$ then there
exists a $\dagger$-selfadjoint  $a \in \frac{1}{2} {\mathfrak F}_A$ with $|1 - a|^2 \leq 1-c$.
Indeed such $a$ can be chosen to also be nearly positive.
\end{proposition} \begin{proof}   As in \cite[Proposition 4.9]{BRord}, but using our  Theorem \ref{havin} (2), there exists
nearly positive $\dagger$-selfadjoint  $a \in \frac{1}{2} {\mathfrak F}_A$ with $$c \leq {\rm Re}(a) \leq 2 {\rm Re}(a) - a^* a,$$
and  $|1-a|^2 \leq 1-c$.
 \end{proof}

We end with an involutive case of the best noncommutative peak interpolation result (from \cite{Bnpi}), a noncommutative generalization of a famous interpolation result of  Bishop.
See \cite{Bnpi} for more context and an explanation of the classical variant, and the significance of the noncommutative variant.
Unfortunately we cannot prove this result for operator $\ast$-algebras without imposing a further strong condition ($d$ commutes with
$b$ and $q^\perp (A^1)^{**} q^\perp \cap A^1$).    This is a good example  of a complicated result 
which  is not clear in advance whether it has  `involutive variants'.   In this case  it is valid, without the strong condition just mentioned, for  involutions of   types  (3) and (4) at the start of Section
{\rm \ref{inv}}.     We treat the type (3) case.   For an operator algebra $A$, let $\bar{A}$ be as in the Remark after Proposition 
\ref{ciop}.  In this case a conjugate linear completely isometric 
involution $\pi$  on $A$ of type {\rm (3)  at the start of Section
{\rm \ref{inv}},  gives rise after composition with the canonical map $- : A \to \bar{A}$, to a
linear completely isometric  isomorphism $A \to \bar{A}$.   
This map extends to a $*$-isomorphism $C^*_{\rm max}(A) \to C^*_{\rm max}(\bar{A}) = (C^*_{\rm max}(A)^\circ)^\star$.
Composing this  with the canonical map $-$, we obtain a conjugate linear $*$-automorphism on $B = C^*_{\rm max}(A)$ which we 
will also write as $\pi$.   This is the compatible conjugate linear involution on $B$.

\begin{theorem}  Suppose that $A$ is an operator algebra, with a conjugate linear completely isometric 
involution $\pi$   of type {\rm (3)}  at the start of Section
{\rm \ref{inv}}.   Suppose that $A$ is a subalgebra of a unital $C^*$ -algebra $B$ with compatible conjugate linear $*$-automorphism  $\pi$ on $B$.
Suppose that $q$ is a closed projection in $B^{**}$ which lies in $(A^1)^{\perp\perp}$ and satisfies $\pi^{**}(q) = q$.   If $b$ is an element in $A$ with $bq=qb$ 
and $b = \pi(b),$ and if $qb^*bq\leq qd$ for an invertible positive $d\in B$ with $d=\pi(d)$ which commutes with $q,$ then there exists a $g\in\ball(A)$ with $gq=qg=bq, g = \pi(g),$ and $g^*  g\leq d.$  
\end{theorem}
\begin{proof}    
By the proof of  \cite[Theorem 3.4]{Bnpi}, 
there exists $h\in A$ with 
$qh=hq=bq,$ and $h^*h\leq d$.  (We remark that $f = d^{-{\frac{1}{2}}}$ in that proof.)
Thus also $\pi(h^*h)
 \leq \pi(d) = d$.
 Let $g=\frac{h+\pi(h)}{2}.$  Then $g = \pi(g)$
and $qg=gq=bq$.  Also 
$$g^* g \leq (\frac{h+\pi(h)}{2})^*(\frac{h+\pi(h)}{2})+(\frac{h-\pi(h)}{2})^*(\frac{h-\pi(h)}{2})
.$$  Thus 
$$g^* g \leq \frac{h^*h}{2}+\frac{\pi(h)^*\pi(h)}{2}= \frac{h^*h}{2}+\pi (\frac{h^*h}{2}) \leq d,$$ as desired.
\end{proof}

As we said in the introduction, a theorem such as the last one about  conjugate linear completely isometric 
involutions is often simply a result about real  operator algebras in the sense of e.g.\ \cite{Sharma}.  In this case the result is:  Suppose that $A$ is a real operator algebra, a subalgebra of a unital real $C^*$ -algebra $B$.
Suppose that $q$ is a closed projection in $B^{**}$ which lies in $(A^1)^{\perp\perp}$.   If $b$ is an element in $A$ with $bq=qb$ 
and  $qb^*bq\leq qd$ for an invertible positive $d\in B$ which commutes with $q,$ then there exists a $g\in\ball(A)$ with $gq=qg=bq,$ and $g^*  g\leq d.$

We end by listing a couple of other results in our paper that do not seem to extend       to all of the   types of involutions discussed at the start of Section
 {\rm \ref{inv}}.   Namely, the parts of Lemma \ref{dcai} involving left and right do not work for involutions
of types (3) and (4).  Also the correspondence between left and right multipliers, left and right  $M$-ideals,
and left and right $\dagger$-ideals, do not work for these types of involutions.

\medskip  

{\bf Acknowledgments.}   This project grew out of 
\cite{BKM}, and we thank Jens Kaad and Bram Mesland for several  ideas and perspectives
learned there.   
We also thank Stephan Garcia--whose work on complex symmetric operators
has influenced some results in our paper--for helpful conversations, and also Elias Katsoulis.

\end{document}